\documentclass[11pt,leqno]{article}
\usepackage{amsthm,amsfonts,amssymb,amsmath,oldgerm}
\usepackage{epsfig}
\numberwithin{equation}{section} 

\newcommand{\btheta}{{\bar \theta}}

\renewcommand\d{\partial}

\def\t{\tau}

\def\eps{\varepsilon }
\def\e{\varepsilon} 

\renewcommand\d{\partial}

\newcommand\R{\mathbb R}

\def\t{\tau}

\def\eps{\varepsilon}
\def\e{\varepsilon}

\newcommand\br{\begin{remark}}
\newcommand\er{\end{remark}}
\newcommand\bp{\begin{pmatrix}}
\newcommand\ep{\end{pmatrix}}
\newcommand\be{\begin{equation}}
\newcommand\ee{\end{equation}}
\newcommand\ba{\begin{equation}\begin{aligned}}
\newcommand\ea{\end{aligned}\end{equation}}

\newcommand{\bap}{\begin{app}}
\newcommand{\eap}{\end{app}}
\newcommand{\begs}{\begin{exams}}
\newcommand{\eegs}{\end{exams}}
\newcommand{\beg}{\begin{example}}
\newcommand{\eeg}{\end{exaplem}}
\newcommand{\bpr}{\begin{proposition}}
\newcommand{\epr}{\end{proposition}}
\newcommand{\bt}{\begin{theorem}}
\newcommand{\et}{\end{theorem}}
\newcommand{\bc}{\begin{corollary}}
\newcommand{\ec}{\end{corollary}}
\newcommand{\bl}{\begin{lemma}}
\newcommand{\el}{\end{lemma}}
\newcommand{\bd}{\begin{definition}}
\newcommand{\ed}{\end{definition}}
\newcommand{\brs}{\begin{remarks}}
\newcommand{\ers}{\end{remarks}}

\newtheorem{theo}{Theorem}[section]

\newtheorem{exam}[theo]{Example}

\newtheorem{exams}[theo]{Examples}

\numberwithin{equation}{section}

\newcommand{\CalT}{\mathcal{T}}

\newcommand{\RR}{{\mathbb R}}
\newcommand{\BbbR}{{\mathbb R}}

\newcommand{\const}{\text{\rm constant}}

\newcommand{\Span}{{\rm Span }}

\newtheorem{theorem}{Theorem}[section]
\newtheorem{proposition}[theorem]{Proposition}
\newtheorem{corollary}[theorem]{Corollary}
\newtheorem{lemma}[theorem]{Lemma}
\newtheorem{definition}[theorem]{Definition}

\newtheorem{example}[theorem]{Example}
\newtheorem{remark}[theorem]{Remark}

\newcommand\cU{{\cal  U}}
\newcommand\cH{{\cal  H}}

\newcommand\cV{{\cal  V}}
\newcommand\cW{{\cal  W}}
\newcommand\cC{{\cal  C}}

\newcommand\cR{{\cal  R}}
\newcommand\cG{{\cal  G}}
\newcommand\cK{{\cal  K}}

\newcommand\cM{{\mathcal M}}

\newcommand\cZ{{\cal  Z}}

\title{
Stability of detonation profiles in the ZND limit }

\author{\sc \small 
Kevin Zumbrun\thanks{Indiana University, Bloomington, IN 47405;
kzumbrun@indiana.edu:
Research of K.Z. was partially supported
under NSF grants no. DMS-0300487 and DMS-0801745.
 }}
\begin{document}

\maketitle


\begin{abstract}
Confirming a conjecture of Lyng--Raoofi--Texier--Zumbrun,
we show that stability of strong detonation waves in the ZND,
or small-viscosity, limit is equivalent to stability of the limiting 
ZND detonation together with stability of the viscous profile 
associated with the component Neumann shock.
More, on bounded frequencies the 
nonstable eigenvalues of the viscous detonation wave converge
to those of the limiting ZND detonation, 
while on frequencies of order one over viscosity, they converge
to one over viscosity times those of the associated viscous Neumann shock.
This yields immediately a number of examples of instability and Hopf
bifurcation of reacting Navier--Stokes detonations
through the extensive numerical studies of ZND stability in the detonation
literature.
\end{abstract}

\tableofcontents



\bigbreak
\section{Introduction}

In one-dimensional, Lagrangian coordinates, the 
reactive Navier--Stokes (rNS) equations modeling reacting flow for a
one-step reaction may be written in abstract form as
\ba \label{rNS}
u_t + f(u)_x &= \eps (B(u)u_{x})_x +kq\varphi(u)z,\\
z_t &= \eps (C(u,z)z_x)_x -k\varphi(u)z,
\ea
where $u$, $f$, $q\in \BbbR^n$, $B\in \BbbR^{n\times n}$,
$z$, $k$, $C$, $\varphi \in \BbbR^1$, and $k,\, \eps>0$ 
\cite{Z1,LyZ1,LyZ2,LRTZ,TZ4}.
Here, $u$ comprises the gas-dynamical variables of
specific volume, particle velocity, and total energy; $z$
measures mass fraction of unburned reactant or, more generally,
``progress'' of a single reaction involving multiple reactants \cite{FD,LyZ1};
$\varphi(u)$ is an ``ignition function'', monotone
increasing in temperature, and usually
assumed for fixed density to be zero below a certain ignition temperature 
and positive above;
$q$ comprises quantities produced in reaction, 
in particular heat released; and $k$ corresponds to reaction rate.
Coefficients $B$ and $C$ model transport effects of, respectively,
viscosity and heat conduction, and species diffusion,
and $\eps$ measures
relative size of transport vs. reaction coefficients, 
typically quite small.

A {\it right-going viscous strong detonation wave} is a 
smoth traveling-wave solution
\be\label{profile}
(u,z)(x,t)=(\bar u, \bar z)(x-st), 
\quad
\lim_{x\to \pm \infty} (\bar u,\bar z)(x)=(u_\pm, z_\pm)
\ee
of solutions of \eqref{rNS} with speed $s>0$
connecting a burned state on the left to an unburned state on the right,
\be\label{zcond}
z_- = 0, \quad z_+ =1,
\ee
with necessarily
\begin{equation} \label{phicond}
\phi(u_-)>0,\quad \phi(u_+)=\phi'(u_+)=0,
 \end{equation}
and satisfying the extreme Lax characteristic conditions
\be\label{Lax}
a_{1}^-< \dots <a_{n-1}^-<s<a_n^-,
\quad
a_{1}^+<\dots<a_n^+ <s,
\ee
where $a_j^\pm$ denote the eigenvalues of
$df(u_\pm)$, ordered by increasing real part.
Left-going viscous detonation waves satisfy symmetric conditions
obtained by reflection, $x\to -x$.

Multi-step reactions may be modeled by the same equations
with vectorial reaction variable $z\in \BbbR^m$, 
and coefficients $q$, $C$, $\varphi$, $k$
modified accordingly.
Likewise, the functions $f$, $\phi$, $B$ may be modified to
depend, more realistically, also on $z$, reflecting the different
chemical makeup of the gas after reaction, with no essential change
at a mathematical level.
For further discussion, see, e.g., \cite{CF,FD,GS,Z1,LyZ1,LyZ2,LRTZ,TZ4,HuZ2}.

A standard simplification in detonation theory is to neglect the
small constant $\eps$ and consider instead the formal $\eps=0$ limit,
or Zeldovich--von Neumann-Doering (ZND) model 
\ba \label{ZND}
u_t + f(u)_x &= kq\varphi(u)z,\\
z_t &= -k\varphi(u)z.
\ea
Indeed, there is by now a tremendous body of literature on this model;
see for example \cite{CF,FD,FW,Er1,Er2,LS,BMR,AT,AlT,KS} and references therein.
The corresponding object to a viscous detonation wave
for the (ZND) model is a {\it right-going strong 
ZND detonation} 
$\bar u^0$ 
of form \eqref{profile}--\eqref{phicond} satisfying \eqref{ZND},
smooth except at a single shock
discontinuity at (without loss of generality) $x=0$,
known as a ``Neumann shock'' \cite{CF,M,GS,LyZ1,LyZ2}, where
$u$ jumps from $u_*$ to $u_+$ as $x$ crosses zero from left to right,
with $z\equiv 1$. 
We have the intuitive picture \cite{CF} of a shock, or ``reaction spike'',
compressing a quiescent mixture and heating it to ignition point, followed
by a slow ``reaction tail'' in which the reaction proceeds until all
reactant is burned, while, meanwhile, $u$ varies from $u_*$ to $u_-$.

A ZND detonation profile is determined implicitly \cite{CF,HuZ2} 
by the property, obtained by integrating the traveling-wave ODE  
\ba \label{ZNDODE}
-su' + f(u)' &= kq\varphi(u)z,\\
-sz' &= -k\varphi(u)z
\ea
and adding $q$ times the second equation to the first,
that $-s(u+qz) + f(u)\equiv \const$, which, together with
$z=1$ for $x\ge 0$ and $z(-\infty)=0$ implies that
\be\label{ZNDsolve}
-su_+ +f(u_+)= -su_* +f(u_*)= -s(u_- - q) +f(u_-),
\ee
giving a unique $u_-$ and profile $\bar u^0$, $x\le 0$, 
for each Neumann shock $(u_*,u_+)$ of speed $s$,
so long as $df(u)-sI$ remains invertible 
for all $0\le z\le 1$ along the curve determined
by 
\be\label{sol}
-s(u+qz) + f(u)\equiv -s(u_+) +f(u_+).
\ee
For, solving \eqref{sol} for $u=u(z)$ by the Implicit Function Theorem
then yields the profile on $x\le 0$ by solution of the second equation
$z'=(-k/s)\phi(u(z))z$, a scalar equation with nonvanishing righthand
side, so long as $u$ remains in the region for which $\phi(u)>0$.

A natural question is the relation between the formally limiting
(ZND) equations and the behavior of the full (rNS) equations as $\eps\to 0$.
At the level of existence of detonation profiles, 
this was investigated by Majda \cite{M} for a simplified model
with $u$, $z\in \RR^1$ and $C\equiv0$ using direct, planar phase portrait
analysis, and extended to the physical (rNS) model
by Gardner \cite{G} using Conley index techniques and
Gasser--Szmolyan \cite{GS} by geometric singular perturbation theory.
More recently, Williams \cite{W} has revisited the existence problem
using more quantitative singular perturbation methods, generating detailed
matched asymptotic expansions to all orders.
In each case, with varying levels of detail, the result is that for each
strong detonation profile $(\bar u^0,\bar z^0)$ of the (ZND) model with 
physical choice of $f$, there exists a family 
$(\bar u^\eps,\bar z^\eps)$ of strong detonation profiles converging
away from $x=0$ as $\eps\to 0$ to $(\bar u^0,\bar z^0)$, and
near $x=0$ to a viscous shock profile for the associated Neumann shock,
in microscopic variables $\tilde x=x/\eps$.

In the present paper, using singular perturbation/asymptotic 
Evans function techniques developed 
in \cite{PZ,HLZ,CHNZ,BHZ,HLyZ1,HLyZ2,OZ1,Z3},
we investigate the {\it stability} of profiles
$(\bar u^\eps, \bar z^\eps)$ in the ZND limit $\eps\to 0$ for
a class of models \eqref{rNS} including both the Majda model\footnote{
Strictly speaking, a variant \cite{L} with nonzero $z$-diffusion $C>0$;
however, the extension to the original Majda model is straightforward,
substituting the weighted norm analysis of Sattinger \cite{Sa} 
for the pointwise analysis of \cite{LRTZ} in order to 
conclude nonlinear stability.
}
studied in \cite{M,L,RV,LyZ2,LRTZ}
and the physical (rNS) equations studied in \cite{G,GS,W,Z1,LyZ1,JLW,TZ4}.
Our conclusion, confirming a conjecture of \cite{LRTZ}, is that 
(linear and nonlinear) 
{\it stability in the ZND limit is equivalent to viscous stability of
the component Neumann shock profile together with hyperbolic stability 
of the associated ZND detonation.}

Together with the results of \cite{HLyZ1} verifying viscous stability
of ideal gas shocks, this gives a rigorous connection between viscous
stability of (rNS) detonations and inviscid stability of the associated
(ZND) detonations, yielding immediately a number of stability and
bifurcation results through the extensive (ZND) literature.
Specialized to the Majda model, it recovers the sole previous result,
due to Roquejoffre and Vila \cite{RV}.

\subsection{Assumptions}

Loosely following \cite{Z1,Z2,MaZ3,MaZ4,LRTZ,TZ4}, we make the assumptions:

\quad (H0) \quad  $f$, $B$, $\phi$, $C\in C^{2}$.
%
\medbreak
\quad (H1)  \quad 
The eigenvalues of $df(u)$ are real, distinct, and 
different from $s$,
for all $u$ near the image
of ZND profile $\bar u^0$, in particular for $u=u_-,u_*,u_+$.
\medbreak
\quad (H2)  \quad 
$B=\begin{pmatrix} 0 & 0 \\ 0 & b \end{pmatrix}$, with
$\Re \sigma b $, $\Re \sigma C \ge \theta>0$,
and 
$df= \begin{pmatrix} df_{11} & df_{12} \\ df_{21} & df_{22} \end{pmatrix}$, 
$df_{11}$ and $df_{12}$ constant,
with the eigenvalues of $df_{11}$ real, semisimple, 
and of one sign relative to $s$, for all $u$ under consideration (i.e.,
near a given detonation profile).
\medbreak
\quad (H3)  \quad 
$\Re \sigma \Big( -i\xi df(u) -\xi^2 B(u)\Big)\le \frac{-\theta \xi^2}{1+\xi^2}$,
$\theta>0$, for all $\xi\in \RR$, and for all $u$ near the image
of ZND profile $\bar u^0$, in particular for $u=u_-,u_*,u_+$.

\br\label{H4mk}
By block upper-triangular structure, we obtain from (H3) also
\be\label{H4}
\Re \sigma \Big( -i\xi \begin{pmatrix}df(u) & 0 \\ 0 & 0\end{pmatrix}
 -\xi^2 
\begin{pmatrix}B(u) & 0 \\ 0 & C(u,z)\end{pmatrix}
+
\begin{pmatrix} 0  & qk\phi(u) \\ 0 & -k\phi(u)\end{pmatrix}
\Big)\le \frac{-\theta \xi^2}{1+\xi^2},
\ee
$\theta>0$, for all $\xi\in \RR$, and for all $u$ near the image
of ZND profile $\bar u^0$, in particular for $u=u_-,u_*,u_+$,
an assumption in the nonlinear stability/bifurcation 
analysis of \cite{TZ4}.
\er


Regarding connecting profiles, we make the further assumptions:

\medbreak
\quad (P1)  \quad 
There exists a ZND profile $u(x,t)=\bar u^0(x-st)$ of \eqref{ZND},
smooth for $x\gtrless 0$, with $\bar u^0(0^-=u_*)$ and 
$\bar u^0(0^+=u_+)$, that is transversal in the sense that
$df(\bar u^0)-sI$ is invertible for all $x$, so that the profile
is locally unique by \eqref{ZNDsolve}.
\medbreak
\quad (P2)  \quad 
There exists a viscous Neumann shock profile 
\be\label{Neumannprof}
u(x,t)=\hat u\Big(\frac{x-st}{\eps}\Big),
\qquad
\lim_{x\to - \infty}\hat u(x)=u_*, \quad 
\lim_{x\to + \infty}\hat u(x)=u_+
\ee
of the associated nonreacting Navier-Stokes equations
$u_t + f(u)_x = \eps (B(u)u_{x})_x$, i.e., a connection
between $u_*$, $u_+$
of the traveling-wave ODE $B(\hat u)\hat u'= f(\hat u)-f(u_+)-s(\hat u-u_+)$,
that is transversal in the sense that $df_{11}(\hat u)-s$ (constant
by assumption (H2)) is invertible.\footnote{ 
This implies that the traveling wave ODE is nondegenerate type;
the profile is then necessarily transversal 
by the extreme shock assumption \eqref{Lax} \cite{MaZ3}.}
\medbreak
\quad (P3)  \quad 
For $\delta>0$ fixed and $\eps>0$ sufficiently small, there exist 
viscous detonation profiles $\bar u^\eps$ of \eqref{rNS}, \eqref{profile}
satisfying for some $C,\, \theta>0$, and $0\le k\le 2$,
\be\label{zndapprox}
|\partial_x^k\big((\bar u^\eps,\bar z^\eps)-(\bar u^0, \bar z^0)\big)|(x)\le C\eps e^{-\theta |x|}
\quad \hbox{\rm for}\; x\le -\delta,
\ee
\be\label{neumannapprox-}
|\partial_x^k\big((\bar u^\eps(x),\bar z^\eps)-(\hat u(x/\eps),1)\big)|\le C\eps
+ C\eps^{1-k} e^{-\theta |x|/\eps} 
\quad \hbox{\rm for}\; -\delta\le x\le  0,
\ee
and
\be\label{neumannapprox+}
|\partial_x^k\big((\bar u^\eps(x),\bar z^\eps)-(\hat u(x/\eps),1)\big)|
\le C\eps^{1-k} e^{-\theta |x|/\eps}
\quad \hbox{\rm for}\; x\ge  0.
\ee

\medbreak

\begin{exam}\label{idealeg}
The physical single-species reactive compressible Navier--Stokes equations, 
in Lagrangian coordinates, are \cite{Ch,TZ4}
\begin{equation}\label{rNSeg}
\left\{ \begin{aligned}
 \d_t \tau - \d_x u & = 0,\\
 \d_t u + \d_x p  & = \d_x (\nu\tau^{-1} \d_x u),\\
 \d_t E + \d_x(pu) & = 
\d_x \big(qd \t^{-2} \d_x z+
\kappa \tau^{-1} \d_x T + \nu\tau^{-1} u \d_x u\big),\\
 \d_t z + k \phi(T) z & = \d_x (d \t^{-2} \d_x z),\\
\end{aligned}\right.
\end{equation}
where $\tau>0$ denotes specific volume, $u$ velocity, 
$ E = e+ \frac{1}{2} u^2 + qz>0  $
total specific energy, $e>0$ specific internal energy,  
and $0 \leq z \leq 1$ mass fraction of the reactant.
Here, $\nu >0$ is a viscosity coefficient, $\kappa> 0$ and $d>0$
are respectively coefficients of heat conduction and species diffusion, $k > 0$ represents the rate of the reaction, and $q$ is the heat release parameter,
with $q>0$ corresponding to an exothermic reaction and $q<0$ to
an endothermic reaction.  
Finally, $T=T(\tau,e,z)>0$ represents temperature and $p=p(\tau,e,z)$ pressure.
 
Under the standard assumptions of
a reaction-independent ideal gas equation of state,
$ p=\Gamma \tau^{-1} e$,
$T=c^{-1} e$, 
where $c>0$ is the specific heat constant and
$\Gamma$ is the Gruneisen constant, and a smooth ignition function  
 $\phi$ vanishing identically for $T \leq T_i$ and strictly positive for $T > T_i,$ 
it is shown in \cite{MaZ3,TZ4} that each of (H0)--(H3)
are satisfied.  Likewise, (P1)--(P3) have been verified in \cite{HuZ2,W}
in this case.  
\end{exam}

\br\label{redundant}
We expect that (P3) can be shown by an argument like that of \cite{W}
to be a general consequence of (H0)--(H3) and (P1)--(P2)
and not an independent assumption.
Note that (P3) typically requires more regularity than we assume in (H0).
\er

\br\label{semirmk}
In (H1), it is enough that eigenvalues be semisimple.
In the present, spectral, analysis, 
we use only that eigenvalues are real and distinct
from $s$ (used to separate decaying and growing slow modes in
the analysis of Section \ref{IIb}).  
In the linearized and nonlinear stability 
analysis for viscous detonations, one may relax the
strict hyperbolicity assumption of \cite{TZ4} to
semisimplicity, as discussed for the viscous shock case
in \cite{MaZ4,Z1}.
\er

\subsection{Main results}
Recall that, associated with the linearized eigenvalue problems for
ZND and rNS detonations are the {\it Evans--Lopatinski determinant} $D_{ZND}$
and the  {\it Evans determinant } $D^\eps_{rNS}$, 
each analytic on $\Re \lambda \ge -\eta<0$, with zeros corresponding to
normal modes of the respective linear problems;
see Sections \ref{evansZND} and \ref{evansrNS} for
precise definitions.
Likewise, there is an Evans determinant $D_{NS}$ associated with
the linearized eigenvalue problem for the associated viscous Neumann shock 
of the nonreacting Navier--Stokes equations (NS); see Section \ref{II}.

Weak Evans--Lopatinski stability of ZND detonations is defined as nonvanishing
of $D_{ZND}$ on $\Re \lambda>0$ and strong Evans--Lopatinski stability
as nonvanishing on  $\Re \lambda \ge 0$ except for a simple zero
at $\lambda=0$ \cite{Er1,Er2,Z1,JLW}. 
Similarly, weak Evans stability of rNS detonations is defined
as nonvanishing of $D^\eps_{rNS}$ 
on $\Re \lambda > 0$ and strong Evans stability as nonvanishing
on $\Re \lambda \ge 0$ except for a simple zero at $\lambda=0$
\cite{Z1,LyZ1,LyZ2,JLW,LRTZ,TZ4}.
Likewise, weak Evans stability of the associated viscous Neumann shock is
defined as nonvanishing of $D_{NS}$ on $\Re \lambda >0$ and
strong stability as nonvanishing on $\Re \lambda \ge 0$ except
for a simple zero at $\lambda=0$: equivalently, nonvanishing of
$\frac{D_{NS}(\lambda)}{\lambda}$ on $\Re \lambda \ge 0$ \cite{MaZ3,Z1,Z2,Z3}.

The following result established in \cite{LRTZ,TZ4}
equates strong Evans stability with linear and nonlinear stability 
of rNS detonations.
A corresponding result holds for the component viscous Neumann
shock \cite{MaZ4,Z2,R,HR,HRZ,RZ,Z4}.

\bpr[\cite{TZ4}]\label{stab}
Given (H0)--(H3), a viscous detonation profile \eqref{profile}
of \eqref{rNS}
is $L^1\cap L^p\to L^p$ 
linearly orbitally stable, $p\ge 1$, if and only if
it is strongly Evans stable,
in which case it is $L^1\cap H^3\to L^p\cap H^3$ asymptotically 
orbitally stable, for $p > 1,$  with 
\ba\label{nonest}
|(\tilde u,\tilde z)(\cdot, t)-(\bar u,\bar z)(\cdot - st-\alpha(t))|_{L^p}
&\le C|(\tilde u_0,\tilde z_0)-(\bar u,\bar z)|_{L^1\cap H^3}(1+t)^{- \frac{1}{2}(1-\frac{1}{p})},\\
|\alpha(t)| &\le C|\tilde U^\e_0-\bar U^\e|_{L^1\cap H^3},\\
|\dot \alpha(t)| &\le C|\tilde U^\e_0-\bar U^\e|_{L^1\cap H^3}(1+t)^{-\frac{1}{2}}
\ea
for some $\alpha(\cdot)$,
where $(\tilde u,\tilde z)$ is the solution of \eqref{rNS} with initial data
$(\tilde u_0,\tilde z_0).$ 
\epr

\begin{proof}
This was established in Theorem 1.12, \cite{TZ4},
for the case described in example \ref{idealeg}.
However, the proof relied only on (H0)--(H3) 
and the consequent \eqref{H4},
hence extends to the general case.
\end{proof}

Our main theorem is the following result linking Evans stability
of (rNS) profiles $(\bar u^\eps,\bar z^\eps)$
in the limit as $\eps\to 0$
with Evans--Lopatinski stability of the limiting (ZND)
profile $(\bar u^0,\bar z^0)$.


\bt\label{main}
Assuming (H0)--(H3), (P1)--(P3),
weak Evans stability of $(\bar u^\eps, \bar z^\eps)$ 
for all $\eps>0$ sufficiently small
implies weak Evans stability of
the viscous profile of the component Neumann shock 
together with weak Lopatinski stability of the limiting ZND 
detonation $(\bar u^0, \bar z^0)$, while strong Evans stability
of $(\bar u^\eps, \bar z^\eps)$ 
for all $\eps>0$ sufficiently small
is implied by strong Evans stability of
the viscous profile of the component Neumann shock 
together with strong Lopatinski stability of the limiting ZND 
detonation $(\bar u^0, \bar z^0)$.

More precisely, (i) For $\eps>0$ and $\eta>0$ sufficiently small, 
there are no zeros of $D_{rNS}$ 
on $\Re \lambda\ge -\eta$ for $|\lambda|\ge C/\eps$, $C$ sufficiently large.
(ii) For $C\le |\lambda| \le C/\eps$, $C$ sufficiently large,
on $\Re \lambda>-\eta$ for $\eta$, $\eps>0$ sufficiently small, 
$\eps$ times each zero of $D^\eps_{rNS}$ converges
to a zero of $\frac{D_{NS}(\lambda)}{\lambda}$ on $\Re \lambda \ge 0$;
moreover, each zero of $D_{NS}$ on $\Re \lambda>0$ is the limit of $\eps$
times a zero of $D_{rNS}$  on $\Re \lambda>0$, for $C\le |\lambda| \le C/\eps$.
(iii) For $|\lambda|\le C_0$, $C_0$ arbitrary, on $\Re \lambda \ge -\eta<0$,
the set of zeros of $D^\eps_{rNS}$ converges as $\eps\to 0$
to the set of zeros of $D_{ZND}$, for any sufficiently small
$\eta>0$ such that $D_{ZND}$ does not vanish for $\Re \lambda=-\eta$
and $|\lambda|\le C_0$.
\et

\begin{proof}
Assertions (i)--(iii) are established in 
Proposition \ref{pIII}, Corollary \ref{sII}, and
Corollary \ref{sI}, whence the remaining assertions follow
by definition of weak and strong stability of the various waves.
\end{proof}

\br\label{erprmk}
Assuming Evans stability of the associated viscous Neumann shock, 
we recover from (ii)--(iii), taking $C_0\to \infty$ and $\eps\to 0$,
the somewhat delicate result of \cite{Er1,Er2} 
that $D_{ZND}$ does not vanish on $\Re \lambda\ge 0$ 
for $|\lambda|$ sufficiently large.
\er

\br\label{JLWrmk}
Recall \cite{JLW} that, for fixed $\eps$, 
low-frequency Evans stability, defined as strong Evans
stability for $|\lambda|\le c_0$, some $c_0>0$, is
equivalent for either ZND or rNS detonations to the
simpler condition of ``Chapman--Jouget'' stability; 
see \cite{JLW} for further details.
Thus, low-frequency stability of ZND waves is necessary
for Evans stability of rNS detonations,
along with weak Evans--Lopatinski stability 
as stated in Theorem \ref{main}.
\er

\subsection{Discussion and open problems}

Together, Proposition \ref{stab} and Theorem \ref{main}
give a rigorous connection between Evans--Lopatinski stability
of (ZND) detonations and nonlinear stability of nearby (rNS) detonations
for $\eps>0$ sufficiently small, 
giving a satisfying mathematical validation 
of physical conclusions made through
the extensive ZND stability studies in detonation literature.
By contrast, to our knowledge there is no analog of Proposition \ref{stab}
for the (ZND) equations themselves, and indeed the physical meaning
of Evans--Lopatinski stability in that context in terms of
nonlinear stability or well-posedness is unclear; 
see \cite{JLW} for further discussion.

More, convergence of zeros of $D^\eps_{rNS}$ to those of $D_{ZND}$
implies that finer phenomena such as Hopf bifurcation are inherited
in the ZND limit as well as stability.
This is perhaps more important, as it is well-known that
ZND detonations are frequently unstable, bifurcating to pulsating
and cellular fronts.
See \cite{TZ2,TZ3,TZ4,SS,BeSZ} 
for a rigorous discussion of such bifurcations in the context of (rNS).

Existence, stability, and bifurcation 
of detonations away from the ZND limit are important open problems.
Such general situations appear to require numerical investigation,
as is standard in the combustion literature even for the simpler
(ZND) model; see, e.g., \cite{LS,HuZ2} and references therein.
Treatment of the ZND limit in multi-dimensions is another important
open problem.
In particular, the implications for nearby (rNS)
profiles of high-frequency ZND instabilities
pointed out in \cite{Er3} is an intriguing mathematical puzzle;
see \cite{JLW} for further discussion.

We note that one-dimensional stability in the ZND limit was previously
established in \cite{RV} for detonation profiles of Majda's model.
Our results both recover and illuminate this prior result,
since the associated Neumann shock profile, because it is  
scalar, is in this case always stable (see, e.g., \cite{Sa}), as is
the limiting ZND detonation.
Related singular perturbation results for systems of conservation laws
may be found in \cite{PZ,HLZ,CHNZ,HLyZ1,HLyZ2,BHZ}
and (for multi-wave patterns) in \cite{OZ1,Z3}.
In particular, we rely heavily on the basic methods of analysis
developed in \cite{PZ}, \cite{HLZ}, \cite{Z3}, and \cite{BHZ}.

\section{The Evans--Lopatinski determinant for (ZND)}\label{evansZND}

We begin by recalling the linearized stability
theory for ZND detonations following \cite{Er1,Z1,JLW,HuZ2}. 
Shifting to coordinates $\tilde x=x-st$ moving with the background
Neumann shock, write \eqref{ZND} as
\begin{equation}\label{abznd}
\begin{aligned}
W_t + F(W)_x=R(W),
\end{aligned}
\end{equation}
where
\begin{equation}\label{abcoefs}
\begin{aligned}
W:=\begin{pmatrix} u\\z \end{pmatrix}, \quad
F:=\begin{pmatrix} f(u)-su\\-sz \end{pmatrix},\quad
R:=\begin{pmatrix} qkz\phi(u)\\-kz\phi(u) \end{pmatrix}.\\
\end{aligned}
\end{equation}

To investigate solutions in the vicinity of a discontinuous 
detonation profile, we postulate existence of a single shock
discontinuity at location $X(t)$, and reduce to a fixed-boundary
problem by the change of variables $x\to x-X(t)$.
In the new coordinates, the problem becomes
\begin{equation}\label{transformed}
\begin{aligned}
W_t + (F(W) - X'(t) W)_x=R(W), \quad x\ne 0,
\end{aligned}
\end{equation}
with jump condition
\begin{equation}\label{transformedRH}
\begin{aligned}
X'(t)[W] - [F(W)]=0, 
\end{aligned}
\end{equation}
$[h(x,t)]:=h(0^+,t)-h(0^-,t)$ as usual denoting jump 
across the discontinuity at $x=0$.

\subsection{Linearization}\label{linearization}
In moving coordinates, $\bar W^0$ 
is a standing detonation, hence 
$(\bar W^0, \bar X)=(\bar W^0,0)$ is a steady solution of 
\eqref{transformed}--\eqref{transformedRH}.
Linearizing \eqref{transformed}--\eqref{transformedRH} about
$(\bar W^0,0)$, we obtain the {\it linearized equations} 
\begin{equation}\label{lin}
\begin{aligned}
(W_t - X'(t)(\bar W^0)'(x)) + (AW)_x =EW,
\end{aligned}
\end{equation}
\begin{equation}\label{linRH}
\begin{aligned}
X'(t)[\bar W^0] - [A W]=0, \quad x=0,
\end{aligned}
\end{equation}
where $A:= (\partial/\partial W)F$, $E:= (\partial/\partial W)R$. 

\medskip
\subsection{Reduction to homogeneous form}\label{homogeneous}

As pointed out in \cite{JLW}, it is convenient for the stability
analysis to eliminate the front from the interior equation \eqref{lin}.
Therefore, we reverse the original transformation to linear order
by the change of dependent variables
\begin{equation}\label{wsharp}
W\to W- X(t)(\bar W^0)'(x),
\end{equation}
motivated by the calculation 
$W(x-X(t),t))-W(x,t)\sim X(t)W_x(x,t)
\sim X(t) (\bar W^0)'(x)$,
approximating to linear order the original, nonlinear transformation.
Substituting \eqref{wsharp} in \eqref{lin}--\eqref{linRH}, and
noting that $x$-differentiation of 
the steady profile equation $F(\bar W^0)_x=R(\bar W^0)$ gives
$(A(\bar W^0)(\bar W^0)'(x))_x=E(\bar W^0)(\bar W^0)'(x)$,
we obtain modified, {\it homogeneous} interior equations
\begin{equation}\label{flin}
\begin{aligned}
W_t + (AW)_x =EW
\end{aligned}
\end{equation}
agreeing with those that would be obtained by a naive
calculation without consideration of the front,
together with the modified jump condition
\begin{equation}\label{flinRH}
\begin{aligned}
X'(t)[\bar W^0]-  [A \big(W+ X(t) (\bar W^0)'\big) ]=0 
\end{aligned}
\end{equation}
correctly accounting for front dynamics.

\medskip

\subsection{The stability determinant}
Seeking normal mode solutions $W(x,t)=e^{\lambda t}W(x)$,
$X(t)=e^{\lambda t}X$,
$W$ bounded, of the linearized equations \eqref{flin}--\eqref{flinRH},
we are led to the generalized eigenvalue equations
\be\label{eigW}
(AW)' = (-\lambda I   + E)W, \quad x\ne 0,
\ee
$$
 X(\lambda[\bar W^0]-[A (\bar W^0)']) - [A W]=0, 
$$
where ``$\prime$'' denotes $d/dx$. or, setting $Z:=AW$, to
\be\label{eig}
Z' = GZ, \quad x\ne 0,
\ee
\begin{equation}\label{eigRH}
\begin{aligned}
X(\lambda[\bar W^0]-[A (\bar W^0)']) - [Z]=0, 
\end{aligned}
\end{equation}
with
\begin{equation}\label{G0}
G:=(-\lambda I  + E)A^{-1},
\end{equation}
\medskip
where we are implicitly using the fact that $A$ is
invertible, by (P1) and $s>0$.

\begin{lemma}[\cite{Er1,Er2,JLW}] \label{lem:splitting} 
On $\R \lambda >0$, the limiting $(n+1)\times (n+1)$
coefficient matrices $G_\pm:= \lim_{z\to \pm \infty} G(z)$ 
have unstable subspaces of fixed rank: full rank $n+1$ for
$G_+$ and rank $n$ for $G_-$.
Moreover, these subspaces extend analytically to $\R \lambda \le -\eta< 0$.
\end{lemma}

\begin{proof}
Straightforward calculation using upper-triangular form
 of $G_\pm$ \cite{Er1,Er2,Z1,JLW}. 
\end{proof}

\begin{corollary}[\cite{Z1,JLW}] \label{cor:splitting} 
On $\R \lambda >0$, the only bounded solution of \eqref{eig}
for $x>0$ is the trivial solution
 $W\equiv 0$.  For $x<0$, the bounded solutions consist of
an $(n)$-dimensional
manifold $\Span \{Z_1^+, \dots, Z_{n}^+\}(\lambda,x)$
of exponentially decaying solutions, analytic in $\lambda$ and
tangent as $x\to -\infty$ to the subspace of exponentially decaying
solutions of the limiting, constant-coefficient equations
$Z'=G_-Z$; moreover, 
this manifold extends analytically to $\R \lambda \le -\eta< 0$.
\end{corollary}

\begin{proof}
The first observation is immediate, using the fact that $G$
is constant for $x>0$.
The second follows from standard asymptotic ODE theory, 
using the conjugation lemma of \cite{MeZ1} (see Lemma \ref{conjlem},
Appendix \ref{asymptotic})
together with the fact that $G$ decays exponentially
to its end state as $x\to -\infty$.
\end{proof}

\begin{definition}\label{lopdef}
We define the {\it Evans--Lopatinski determinant}
\begin{equation}\label{evanseq}
\begin{aligned}
D_{ZND}(\lambda)&:=
\det \begin{pmatrix}
Z_1^-(\lambda,0), & \cdots, & Z_{n}^-(\lambda,0), &
\lambda [\bar W^0]-[A (\bar W^0)']\\
\end{pmatrix}\\
&=
\det \begin{pmatrix}
Z_1^-(\lambda,0), & \cdots, & Z_{n}^-(\lambda,0), &
\lambda[\bar W^0]+A(\bar W^0)'(0^-)\\
\end{pmatrix},\\
\end{aligned}
\end{equation}
where $Z^-_j(\lambda,x)$ are as in Corollary \ref{cor:splitting}.
\end{definition}

The function $D_{ZND}$ is exactly the {\it stability function} 
derived in a different form by Erpenbeck \cite{Er1,Er2}.
The formulation \eqref{evanseq} is of the standard form arising
in the simpler context of (nonreactive) shock stability \cite{Er4}.
Evidently (by \eqref{eigRH} combined with Corollary \ref{cor:splitting}), 
$\lambda$ is a generalized eigenvalue/normal mode
for $\R \lambda \ge 0$ if and only if $D_{ZND}(\lambda)=0$.

\section{The Evans determinant for (rNS) }\label{evansrNS}

\subsection{Linearization}
Linearizing \eqref{rNS} about $(\bar u^\eps, \bar z^\eps)$ in moving
coordinates yields linearized eigenvalue equations
\ba\label{eigrNS}
\lambda W + (\tilde AW)_x =EW + (\eps\tilde B W_x)_x,
\ea
where $W=\bp u_1\\u_2\\z\ep=\bp W_1\\W_2\ep$, $W_2=\bp u_2\\z\ep$,
 $\bar W^\eps=\bp \bar u^\eps\\\bar z^\eps\ep$,
 $\tilde B =
\bp 0& 0\\0& \tilde b\ep =
\bp 0 & 0 & 0\\ 0 & b & 0\\ 0&0&C\ep (\bar W^\eps)$,
 and
\be\label{bara}
\tilde A v:= dF(\bar W^\eps)v- \eps(d \tilde B(\bar W^\eps) v)\bar W^\eps_x,
\qquad
(\tilde A_{11}, \tilde A_{12}) = (df_{11}-s,(df_{12},0))\equiv \const,
\ee
$E$, $F$ as in \eqref{flin}.

\subsection{Expression as a first-order system}
Setting  $\tilde x=x/\eps$ and
$$
\cW= \bp Y\\  W_2\ep:=
\bp \tilde AW-\eps \tilde B W_x \\ W_2\ep,
\qquad
Y=\bp Y_1\\Y_2\ep ,\quad
W_2:=\bp u_2\\z\ep,
$$
we may write \eqref{eigrNS} as
a first-order system 
\be\label{1linrNS}
\dot \cW=\cG^\eps(\lambda, x)\cW,
\ee
where
\be\label{coeff}
\cG^\eps:=
\begin{pmatrix}
\eps (E-\lambda)\tilde A_{11}^{-1} &0& 
-\eps(E-\lambda)\tilde A_{11}^{-1}\tilde A_{12}\\
0&0& \eps(E-\lambda)\\
\tilde b^{-1}\tilde A_{21}\tilde A_{11}^{-1}& -\tilde b^{-1}& 
\tilde b^{-1}(\tilde A_{22}-\tilde A_{21}\tilde A_{11}^{-1}\tilde A_{12})
\end{pmatrix}
\ee
and $\dot{ }$ denotes $d/d\tilde x$.
Here, we are using implicitly the facts that $\tilde A_{11}$ and $\tilde b$
are invertible, 
by \eqref{bara}, (P2), and (H2).

We have the following analogs of Lemma \ref{lem:splitting} and
Corollary \ref{cor:splitting}, which follow in essentially the same
way; see \cite{LRTZ,TZ4} for further details.

\begin{lemma}[\cite{Z1,TZ4}] \label{lem:splittingrNS} 
On $\R \lambda >0$, the limiting $(n+1+r)\times (n+1+r)$
coefficient matrices $\cG_\pm:= \lim_{z\to \pm \infty} \cG(z)$,
$r=\dim u_2 +1$, 
have stable subspaces of fixed, equal rank $r$.
Moreover, these subspaces extend analytically to $\R \lambda \le -\eta< 0$.
\end{lemma}

\begin{corollary}[\cite{Z1,TZ4}] \label{cor:splittingrNS} 
On $\R \lambda >0$, the bounded solutions of \eqref{eigrNS}
on $x\le 0$ consist of an $(n+1)$-dimensional manifold 
$\Span \{\cW_1^-, \dots, \cW_{n+1}^-\}(\lambda,x)$
of exponentially decaying solutions in the backward $x$ direction, 
analytic in $\lambda$ and
tangent as $x\to -\infty$ to the subspace of exponentially decaying
solutions in the backward $x$ direction
 of the limiting, constant-coefficient equations $\cW'=\cG_-\cW$;
the bounded solutions of \eqref{eigrNS}
on $x\ge 0$ consist of an $(r-1)$-dimensional manifold 
$\Span \{\cW_{n+2}^+, \dots, \cW_{n+1+r}^+\}(\lambda,x)$
of exponentially decaying solutions, 
analytic in $\lambda$ and
tangent as $x\to +\infty$ to the subspace of exponentially decaying
solutions of the limiting, constant-coefficient equations $\cW'=\cG_+\cW$;
Moreover, these manifolds extend analytically to $\R \lambda \le -\eta< 0$.
\end{corollary}

\subsection{The stability determinant}

\begin{definition}\label{evansdef}
We define the {\it Evans function}
\begin{equation}\label{evanseqrNS}
D^\eps_{rNS}(\lambda):=
\det \begin{pmatrix}
\cW_1^-, & \cdots, & \cW_{n+1}^- , &
\cW_{n+2}^+, & \cdots, & \cW_{n+1+r}^+ \end{pmatrix} (\lambda,0),
\end{equation}
where $\cW^\pm_j(\lambda,x)$ are as in Corollary \ref{cor:splittingrNS}.
\end{definition}

\section{Fast vs. slow coordinates}

Note that the coordinate transformation to stretched, or
``fast'' variables
\be\label{fast}
(\tilde x,\tilde t,\tilde \lambda)=
(x/\eps,t/\eps,\eps \lambda),
\ee
changes equations \eqref{rNS} to
\ba \label{fastrNS}
u_{\tilde t} + f(u)_{\tilde x} &=  (B(u)u_{{\tilde x}})_{\tilde x} +\eps kq\varphi(u)z,\\
z_{\tilde t} &=  (C(u,z)z_{\tilde x})_{\tilde x} -\eps k\varphi(u)z,
\ea
shifting the small parameter $\eps$ from diffusion to reaction terms.
The computations to follow may be thought of as
alternating between (original) slow variables and fast variables
in our analysis, as convenient for different regions in $\lambda$ and $x$.
In particular, first-order equations \eqref{1linrNS}--\eqref{coeff}
may be recognized as the first-order linearization of fast equations 
\eqref{fastrNS}.

\section{Region I: $|\lambda|\le C$}\label{I}

We first study the critical ``ZND'' region
$ |\lambda|\le C$, or $\eps/C\le|\tilde \lambda|\le C\eps$,
where behavior of the (rNS) Evans function $D_{rNS}$ is governed by
that of the Evans--Lopatinski determinant $D_{ZND}$.
Setting $M>>1$ to be a large constant to be determined later,
we study separately the zones $\tilde x=x/\eps \le -M$
and $\tilde x=x/\eps \ge -M$, on which the profile $\bar W^\eps$
is dominated respectively by the (ZND) profile $\bar W^0$ and
the viscous shock profile $\hat W$, as described in (P3).

\subsection{``Slow'', or ``reaction'' zone, $\tilde x\le -M$}\label{slowzone}
Note that $\tilde A$, hence 
$N:=\tilde b^{-1}(\tilde A_{22}- \tilde A_{21}\tilde A_{11}^{-1}\tilde A_{12})$
is invertible for $x\le -\delta$, by \eqref{bara} and (P1)--(P2).
Setting $\cW=T\cZ$, where
$T:=\bp I & 0 \\ -N^{-1}\ell& I\ep$ and 
$\ell:= (\tilde b^{-1}\tilde A_{21}\tilde A_{11}^{-1}, -\tilde b^{-1})$, 
$$
-N^{-1}\ell=
-(\tilde A_{22}- \tilde A_{21}\tilde A_{11}^{-1}\tilde A_{12})^{-1}
(\tilde A_{21}\tilde A_{11}^{-1}, -I), 
$$
and noting that $\dot N,\, \dot \ell \sim \dot{ \bar W}^\eps$ 
on $\tilde x\le -M$, by (P3), we transform 
\eqref{1linrNS}--\eqref{coeff} to
\be\label{cZeq}
\dot\cZ= \cH^\eps \cZ,
\ee 
\be\label{H}
\cH^\eps= T^{-1}\cG^\eps T -T^{-1}\dot T=
\begin{pmatrix} \eps (E-\lambda)\tilde A^{-1} &\eps m\\ 0&N+O(\eps)\\ \end{pmatrix}
+
\bp 0 & 0\\ O(\eps + |\dot{\bar W}^\eps|) & 0\ep,
\ee
where $m:= (E-\lambda)\bp -\tilde A_{11}^{-1}\tilde A_{12}\\ I \ep$,
and terms $O(\cdot)$ are smooth and analytic in $\lambda$.
Here, $N$, by \eqref{Lax} and the relation between viscous and inviscid
shock structure \cite{MaZ3}, has one positive eigenvalue and $r-1$ negative 
eigenvalues, where $r=\dim u_2+1$.

By a further transformation $\cZ=S\cV$,
$S=\bp I & -\eps N^{-1}m\\ 0 & I\ep$, we transform to
$\dot \cV= \cK^\eps \cV$, where
\be\label{cK}
\cK^\eps = S^{-1}\cH^\eps S -S^{-1}\dot S=
\begin{pmatrix} \eps (E-\lambda)\tilde A^{-1} &0\\ 
0&N\\ \end{pmatrix}
+
\bp  O(\eps^2 + \eps|\dot{\bar W}^\eps|)& O(\eps^2 + \eps|\dot{\bar W}^\eps|)\\
\\ O(\eps + |\dot{\bar W}^\eps|) & O(\eps + |\dot{\bar W}^\eps|)\ep.
\ee
By a further transformation $\bp I & 0 \\ 0 & r\ep$ 
if necessary, we may take without loss of 
generality $N$ block-diagonal, $N=\bp N_1&0\\0&N_2\ep$, with
$\Re N_1\ge \eta>0$ and $\Re N_2\le -\eta <0$, both with a uniform
spectral gap from block $\eps(E-\lambda)\tilde A^{-1}\sim \eps$.

Applying the asymptotic ODE results of Lemma \ref{reduction} with
$\eta\sim 1$, $\delta=O(\eps + |\dot{\bar W}^\eps|)$,
and $|\Theta|\le C$, and of Remark \ref{tritrack}.1 with
$\eta\sim 1$,
$\delta:=\eps(\eps+ |\dot{\bar W}^\eps|)$, $|\Theta_{12}|+|\Theta_{11}|\le C$,
$|\Theta_{21}|+|\Theta_{22}|\le C/\eps$, 
we find that there is a change of coordinates
$\cZ=S \cU$, $\cU=\bp \cU_1\\\bp \cU_2\\\cU_3\ep\ep$, defined on 
$\tilde x\le -M$, 
\be\label{Sinv}
S^{-1}=\bp I & -\Phi_1\\ -\Phi_2 & I\ep=
\bp I & \eps O(\eps + |\dot{\bar W}^\eps|)\\ 
O(\eps + |\dot{\bar W}^\eps|) & I\ep
=
\bp I& o(\eps)& o(\eps)\\
 o(1)& I& o(1)\\
 o(1)&  o(1)& I\ep
\ee
close to the identity
converting \eqref{cZeq} into three decoupled equations
\ba\label{decoup}
\dot \cZ_1&=\eps(E-\lambda)\tilde A^{-1}\cZ_1 + 
O(\eps^2+\eps |\dot {\bar W}^\eps|)\cZ_1;
\qquad \cU_2\equiv 0, \, \cU_3\equiv 0, \\
\dot \cZ_2&= (N_1+o(1))\cZ_2; 
\qquad \cU_1\equiv 0, \, \cU_3\equiv 0, \\
\dot \cZ_3&= (N_2+o(1))\cZ_3 ;
\qquad \cU_1\equiv 0, \, \cU_2\equiv 0. \\
\ea

Focusing on the ``slow'' $\cZ_1$ mode,
changing coordinates from $\tilde x$ back to $x=\eps \tilde x$, we obtain
\be\label{Z1eq}
\cZ_1'=(E-\lambda)\tilde A^{-1}\cZ_1 + O(\eps +|\dot{\bar W}^\eps(x/\eps)|)\cZ_1, 
\ee
where ``$'$'' denotes $d/dx$, and all coefficients converge uniformly
as $Ce^{-\eta |x|}$, $\eta>0$ to their limits as $x\to -\infty$.
Applying the convergence lemma, Lemma \ref{evanslimit}
(Appendix \ref{s:conv}),
for $x\le -M$,
together with Remark \ref{varconj}, \eqref{altconv2},
we find that there is a further coordinate change $\cZ_1=PZ$, 
$P=I+O( \eps+ e^{-\eta M})=I+o(1)$ for $|x|\le 2M$, 
taking the decaying/growing modes of \eqref{Z1eq} 
to those of
$$
Z'=(E-\lambda)\tilde A^{-1}Z ,
$$
which may be recognized as exactly
the interior equation \eqref{eig},\eqref{G0} associated with
the Evans--Lopatinski development.

Tracing back through our coordinate transformations, we find
that the $n$ slowly-decaying modes $\cW_1^-, \dots \cW_{n}^-$
as $\tilde x\to -\infty$ are given at $\tilde x=-M$ by
\be\label{slowexp}
\cW_j^-=(I+o(1))\bp Z_j^- \\ * \ep , \quad j=1, \dots, n,
\ee
where $Z_j^-$ are as in \eqref{evanseq}, while the single fast-decaying
mode as $\tilde x\to - \infty$ is given at $\tilde x=-M$ by
\be\label{fastexp}
\cW_{n+1}^-= c(- M)
\bp o(1)\eps v \\ (I+o(1))v 
\ep, 
\ee
where $c(\tilde x)$ is exponentially decaying as $\tilde x\to -\infty$
and $v$ lies in the unique stable eigendirection of $N(-M)$,
$N:=\tilde b^{-1}(\tilde A_{22}- \tilde A_{21}\tilde A_{11}^{-1}\tilde A_{12})$.
Here, $o(1)$ depends on $M$, going to zero as $M\to \infty$
and $\eps\to 0$ at the same time.

\br\label{Nrmk}
At $\eps=0$, 
the single fast-decaying mode reduces to the translational
zero eigenfunction of the associated viscous Neumann shock (necessarily
fast-decaying), and so we may refine \eqref{fastexp} to
\be\label{furtherfastexp}
\cW_{n+1}^-= 
\bp o(1)\eps \dot{\bar W}_2^0 \\ (I+o(1))\dot{\bar W}_2^0 \ep, 
\qquad \dot{\bar W}_2^0=\bp \dot{ \hat u}_2\\ \dot{\hat z}\ep.
\ee
\er

\subsection{``Fast'', or ``Neumann shock'' zone, $\tilde x\ge -M$}\label{fastzone}

Applying the convergence lemma, Lemma \ref{evanslimit}, for $\tilde x\ge -M$,
using the asymptotics of (P3), together with Remark \ref{varconj},
\eqref{altconv2},
we find that there is a change of
coordinates $\hat P^\eps_+$ with  $\hat P^\eps_+=I+ O(\eps )$ for
$|\tilde x|\le M$, such that
$\cW:=\hat P^\eps_+\cV$ converts \eqref{1linrNS}--\eqref{coeff} to the
same equations with $\eps\equiv 0$, i.e., the shock eigenvalue
system at $\lambda=0$, hence, by inspection, the decaying modes
$\cV_{n+2}^+, \dots, \cV_{n+1+r}^+$ at $+\infty$, necessarily fast-decaying,
by the Lax characteristic assumption \eqref{Lax} (see \cite{MaZ3}), 
are by inspection (see \cite{Z1} for similar calculations)
of form 
\be\label{+decayV}
\cV_j^+= \bp 0 \\  v_j(x)\ep,
\ee
$v_j$ independent, hence for $|x|\le M$
\be\label{+decayW}
\cW_j^+= (I+O(\eps ))
\bp 0 \\ v_j\ep,
\qquad j=n+2, n+1+r .
\ee
Here the constant $O(\cdot)$ depends on (growing exponentially with) 
the fixed constant $M$,
since we are shifting $\tilde x \to \tilde x+M$ in order to apply the 
convergence lemma.
Evaluating at at $\tilde x=-M$, we have
\be\label{further+decayW}
\cW_j^+(-M)= (I+ O(\eps))
\bp 0 \\ v_j(-M)\ep,
\qquad j=n+2,\dots, n+1+r .
\ee

\subsubsection{Variation in $\eps$}\label{var}
At this point, gathering information, and noting, by Abel's formula,
that the Wronskian $D_{rNS}(\lambda)$ evaluated at $\tilde x=0$
is equal to a nonzero constant times the same Wronskian evaluated
at the point $\tilde x=-M$ at which we have information about
solutions from both sides, we have that $D_{rNS}(\lambda)$ is
proportional by a nonzero constant, analytic in $\lambda$, $\eps$, to
\ba\label{basicdet}
\det \Big( 
(I+o(1))\bp Z_1^- \\ *\ep,   &\dots , (I+o(1))\bp Z_n^- \\ *\ep,
(I+o(1)(\eps)\bp 0 \\v_{n+1}\ep , \\
&\qquad \qquad
(I+O(\eps))\bp 0 \\v_{n+2}\ep , \dots, 
(I+O(\eps))\bp 0 \\v_{n+1+r}\ep \Big)\\
&\qquad = O(\eps),
\ea
since up to $O(\eps)$ there are only $n$ nonzero entries $Z_j^-$ in the
$(n+1)$-dimensional $Y$ block.
Here, the constants $O(\eps)$ and $o(1)$ by construction are analytic in
$\lambda$, $\eps$ as well.

This reflects the fact that at $\eps=0$ there is a 
solution $W=\partial_{\tilde x}
\hat W=\bp \partial_{\tilde x}{\hat u}\\0\ep$ decaying at both
$\pm \infty$ of \eqref{eigrNS}, corresponding to the translational
eigenmode of the linearized equations about the associated viscous
Neumann shock; see Remark \ref{Nrmk}.
To extract the next-order behavior, we compute the
first variation of this special mode with respect to $\eps$ 
at $\eps=0$,
by an argument similar to those used in
\cite{GZ,ZS,LyZ1,LyZ2} to compute the first variation with
respect to $\lambda$ at $\lambda=0$.

Specifically, recalling that $v_{n+1},\dots, v_{n+1+r}$ are a basis in $\cC^r$,
we may perform a column operation using the final $r$ columns 
to cancel the entry $v_{n+1}$ in the $(n+2)$nd column
in determinant \eqref{basicdet}, to obtain
\ba\label{2det}
D^\eps_{rNS}(\lambda)&=
\psi(\lambda,\eps)
(I+o(1))\eps \det \bp v_{n+2} & \dots & v_{n+1+r}\ep
\det \bp Z_1^- &\dots &  Z_n^- & -Y_\eps \ep \\
\ea
hence 
\be \label{detconv}
\frac{D^\eps_{rNS}}{\eps\Psi(\lambda, \eps)}
= \det \bp Z_1^- &\dots &  Z_n^- & -Y_\eps \ep  + o(1),
\ee
as $\eps\to 0$, $o(1)\to 0$ as $M\to \infty$, 
where $\psi$ and $\Psi:=\psi \det \bp v_{n+2} & \dots & v_{n+1+r}\ep$ 
are nonvanishing factors analytic in $\lambda$, $\eps$, and 
\be\label{Y}
Y_\eps:=\partial_\eps  Y^*|_{\eps=0}=
\partial_\eps  \cW^*_1|_{\eps=0},
\ee
where $\cW^*(\eps,\lambda, x)$ is the (necessarily fast-) decaying
solution of \eqref{eigrNS} at $x\to +\infty$ defined by
\be\label{keydef}
\cW=\hat P_+^\eps(\lambda, x) \bp 0\\v_{n+1}\ep,
\ee
where $\hat P_+^\eps (\lambda, x)$ is the conjugating transformation
described above, and
$$
W^*|_{\eps=0}=\partial_{\tilde x} \hat W =\bp \partial_{\tilde x}{\hat u}\\0\ep.
$$

Writing \eqref{eigrNS} in $\tilde x$ coordinates as
\be\label{tildeeigrNS}
\eps(E- \lambda ) W 
=(\tilde A^\eps W -\tilde B^\eps  W_{\tilde x})_{\tilde x}
= \dot Y,
\ee
and differentiating \eqref{eigrNS} with respect to $\eps$, 
we obtain the variational equations
$$
(E- \lambda ) W^*|_{\eps=0} 
-(\partial_\eps\tilde A^\eps W^* 
-\partial_\eps \tilde B^\eps  W^*_{\tilde x})_{\tilde x}|_{\eps=0}
= \dot Y_\eps,
$$
or
\be\label{vareq}
E \hat W_{\tilde x}
-(\lambda \hat W+ \partial_\eps\tilde A^\eps\hat W_{\tilde x}  
-\partial_\eps \tilde B^\eps  \hat W_{\tilde x \tilde x})_{\tilde x}
= \dot Y_\eps.
\ee

Finally, recall that differentiating the traveling-wave ODE
$$
 (\tilde A\bar W^\eps)_{\tilde x} =\eps E\bar W^\eps 
+ (\tilde B \bar W^\eps_{\tilde x})_{\tilde x}
$$
with respect to $\tilde x$ yields
\be\label{trav}
E\bar W^\eps_{\tilde x} =
 \eps^{-1}(\tilde A\bar W^\eps_{\tilde x} - 
\tilde B \bar W^\eps_{\tilde x \tilde x})_{\tilde x}=
 (\tilde A\bar W^\eps_{x} - \eps \tilde B \bar W^\eps_{ xx})_{\tilde x}.
\ee
Together with the estimates 
\be\label{est1}
|\bar W^\eps_{\tilde x}- \hat W_{\tilde x}|\le C\eps e^{-\eta |\tilde x|},
\ee
$\eta>0$, for $x\ge -M$ and
\be\label{est2}
|\hat W_{\tilde x}(-M)|, \; |\hat W_{\tilde x \tilde x}(-M) |=o(1)
\ee
coming from asymptotics (P2), (P3), 
we find, combining \eqref{vareq}, \eqref{trav}, \eqref{est1}, and \eqref{est2}
that
\ba\label{keyrelation}
 (Y_\eps)_{\tilde x}&=
E (\hat W_{\tilde x}-\bar W^\eps_{\tilde x})+
 (\tilde A\bar W^\eps_{x} - \eps \tilde B \bar W^\eps_{ xx})_{\tilde x}
-(\lambda \hat W+ \partial_\eps\tilde A^\eps\hat W_{\tilde x}  
-\partial_\eps \tilde B^\eps  \hat W_{\tilde x \tilde x})_{\tilde x}\\
&=
 (\tilde A\bar W^\eps_{x} -\lambda \hat W 
- \eps \tilde B \bar W^\eps_{ xx}
 - \partial_\eps\tilde A^\eps\hat W_{\tilde x}  
+\partial_\eps \tilde B^\eps  \hat W_{\tilde x \tilde x})_{\tilde x}
+ O(\eps e^{-\eta |\tilde x|}).
\ea
Integrating \eqref{keyrelation} in $\tilde x$ from $\tilde x=-M$
to $\tilde x=+\infty$ thus yields
\be\label{intcomp}
 -Y_\eps(-M)=
 (-\tilde A\bar W^\eps_{x} +\lambda \hat W 
+ \eps \tilde B \bar W^\eps_{ xx}
 + \partial_\eps\tilde A^\eps\hat W_{\tilde x}  
-\partial_\eps \tilde B^\eps  \hat W_{\tilde x \tilde x})|_{\tilde x=-M}^{\tilde x=+\infty}
+ O(\eps),
\ee

From asymptotics (P3), we find that
\be\label{keyas}
 \tilde A\bar W^\eps_{x} - \eps \tilde B \bar W^\eps_{ xx}
=
 \tilde A\bar W0_{x} 
 +(\tilde A\hat W_{x} - \eps \tilde B \hat W_{ xx})
+o(1).
\ee
Recalling (by differentiation in $x$ of the traveling-wave ODE
for the viscous Neumann shock profile) that
$$
 \tilde A \hat W_{x} - \eps \tilde B \hat W_{ xx}\equiv 0,
$$
we find, substituting \eqref{keyas} into \eqref{intcomp}
and  discarding lower-order terms and vanishing boundary terms at $+\infty$,
\be\label{intresult}
 -Y_\eps(-M)=
\lambda [ \bar W^0] -[ A(\bar W^0)'(0)] + o(1),
\ee
where $\bar W^0$ as in (P1) denotes the associated ZND profile
and $[\cdot]$ the jump in values across its Neumann shock discontinuity.

\subsection{Convergence to $D_{ZND}$}

\bpr\label{pI}
Assuming (H0)--(H3), (P1)--(P3),
for $|\lambda|\le C$ and $\Re \lambda \ge -\eta$, $\eta>0$
sufficiently small, 
$\frac{D^\eps_{rNS}(\lambda)}{\eps\tilde \Psi(\lambda, \eps)}$
converges uniformly as $\eps \to 0$ to $D_{ZND}(\lambda)$,
where $\tilde \Psi(\cdot,\cdot)$ is 
a nonvanishing factor that is analytic in $\lambda$.
\epr

\begin{proof}
Choosing monotone sequences $M_j$ (increasing) and $\eps_j$ (decreasing)
such that $o(1)\le 1/j$ in \eqref{intresult} and \eqref{detconv},
define $\tilde \Psi(\lambda,\eps)$ to be equal to the function
 $\Psi(\lambda,\eps)$ in \eqref{detconv} that is associated
with $M_j$, where $j$ is the maximum integer such that
$\eps \le \eps_j$.
Then, $\tilde \Psi$ is analytic in $\lambda$ by construction, and,
combining \eqref{detconv}, \eqref{intresult}, and the definition of
$\eps_j$, $M_j$, we have
$$
\Big|
\frac{D^\eps_{rNS}(\lambda)}{\eps\tilde \Psi(\lambda, \eps)}-D_{ZND}(\lambda)
\Big| \le C/j \to 0
\quad \hbox{\rm as $\eps \to 0$.}
$$
\end{proof}

\bc\label{sI}
Assuming (H0)--(H3), (P1)--(P3),
for $|\lambda|\le C$ and $\Re \lambda \ge -\eta$, $\eta>0$,
the set of zeros of $D^\eps_{rNS}$ converges as $\eps\to 0$
to the set of zeros of $D^\eps_{rNS}$, for any sufficiently
small $\eta>0$ such
that $D_{ZND}$ does not vanish for $\Re \lambda=-\eta$ and $|\lambda|\le C$.
\ec

\begin{proof}
Noting that zeros of $D^\eps_{rNS}$ agree with zeros of
$\frac{D^\eps_{rNS}(\lambda)}{\eps\tilde \Psi(\lambda, \eps)}$,
we obtain the result by properties of uniform limits of analytic functions.
\end{proof}

\section{Region II: $C/\eps\ge |\lambda|\ge C>>1$}\label{II}

We next consider the ``Neumann shock'' region 
$C\eps \le \tilde \lambda \le C$, $C>>1$,
in which behavior of $D_{rNS}$ is dominated by that of the Evans
function $D_{NS}$ of the associated viscous Neumann shock.
Here, $D_{NS}$ is defined similarly as in Definition 
\ref{evansdef} of $D_{rNS}$ as
\begin{equation}\label{evanseqN}
D^\eps_{N}(\tilde \lambda):=
\det \begin{pmatrix}
\hat \cW_1^-, & \cdots, & \hat \cW_{n+1}^- , &
\hat \cW_{n+2}^+, & \cdots, & \hat \cW_{n+1+r}^+ \end{pmatrix} 
(\tilde \lambda,0),
\end{equation}
where $\hat \cW_j^\pm$ are decaying modes at $\pm \infty$ of 
the linearized eigenvalue equation
\ba\label{eigN}
\tilde \lambda W + (\hat AW)_{\tilde x} = (\hat B W_{\tilde x})_{\tilde x}
\ea
about $\hat W$, written as a first order system 
\be\label{sfirst}
\dot{\hat \cW}=\hat \cG(\tilde \lambda, \tilde x) \hat {\cW}, 
\qquad
\hat \cG:=
\begin{pmatrix}
-\tilde \lambda\hat A_{11}^{-1} &0& 
-\tilde \lambda\hat A_{11}^{-1}\tilde A_{12}\\
0&0& -\tilde \lambda\\
\hat b^{-1}\hat A_{21}\hat A_{11}^{-1}& -\hat b^{-1}& 
\hat b^{-1}(\hat A_{22}-\hat A_{21}\hat A_{11}^{-1}\hat A_{12})
\end{pmatrix},
\ee
with $\hat \cW=\bp \hat Y\\\hat W\ep=
\bp \hat A \hat W- \hat B \dot{\hat W}\\\hat W\ep $, 
$\hat B:=\tilde B(\hat W)$,
and
$\hat A v:= dF(\hat W)v- (d \tilde B(\hat W) v)\bar W_{\tilde x}$,
where $F$ is as in \eqref{flin}.  
For further details, see, e.g., \cite{Z1,Z2}.

\subsection{Fast zone $\tilde x\ge -M$}

Noting that $\tilde \lambda=\eps \lambda$ is bounded
by assumption, we have by \eqref{coeff} and (P3) 
\ba\label{coeffII}
\cG^\eps&=
\begin{pmatrix}
\eps E-\tilde \lambda\tilde A_{11}^{-1} &0& 
-\eps E-\tilde \lambda\tilde A_{11}^{-1}\tilde A_{12}\\
0&0& \eps E-\tilde \lambda\\
\tilde b^{-1}\tilde A_{21}\tilde A_{11}^{-1}& -\tilde b^{-1}& 
\tilde b^{-1}(\tilde A_{22}-\tilde A_{21}\tilde A_{11}^{-1}\tilde A_{12})
\end{pmatrix}
= \hat \cG + O(\eps e^{-\eta|x|}),
\\
\ea
where $\hat G$ is bounded and converges at uniform exponential
rate to its limits at $\pm \infty$.
Applying the convergence lemma, Lemma \ref{evanslimit}, for $\tilde x\ge -M$,
together with Remark \ref{varconj},
\eqref{altconv}, similarly as in Section \ref{fastzone}
but now treating $\tilde \lambda$ as a fixed parameter,
we find that there is a change of
coordinates $\hat P^\eps_+$ with  $\hat P^\eps_+=I+ O(\eps )$ for
$\tilde x \ge -M$, such that
$\cW:=\hat P^\eps_+\hat \cW$ converts \eqref{1linrNS}--\eqref{coeff} to the
viscous shock system \eqref{sfirst}.
Evaluating at $\tilde x=-M$, we thus have
$ \cW_j^+(\lambda, -M)= (I+ O(\eps)) \hat \cW_j^+(\tilde \lambda, -M), $
or, by the assumption that $|\tilde \lambda|>>\eps$,
\be\label{further+decayhatW}
\cW_j^+(\lambda, -M)= \big(I+ o(\tilde \lambda)\big) 
\, \hat \cW_j^+(\tilde \lambda, -M),
\qquad j=n+2, \dots, n+1+r .
\ee

\subsection{Slow zone $\tilde x \le -M$}

\subsubsection{Case a. $C/\eps\ge |\lambda|\ge 1/C\eps$, $C>0$ arbitrary}\label{IIa}
We first treat the easier case $1/C\eps\le |\lambda|\le C/\eps$, or
$ C^{-1}\le |\tilde \lambda|\le C$, for arbitrary $C>0$.
From (H3), it follows that $\hat G(\lambda, -\infty)$ has no pure
imaginary eigenvalues for $\Re \tilde \lambda>0$ and $\lambda\ne 0$.
By continuity, the same holds for $\cG(\lambda, \tilde x)$ for
$\tilde x\le -M$, $\Re \tilde \lambda>0$ and $\tilde \lambda$ bounded
and bounded away from zero.

By the assumption $1/C\le |\tilde \lambda|\le C$, therefore, the
stable and unstable subspaces of $\cG$ have a uniform spectral gap
for all $\tilde x\le -M$, hence, by standard matrix perturbation
theory \cite{K,ZH,Z5,GMWZ5}, there exist smooth 
transformations $T(\cG)$ such that 
$$
T^{-1}\cG T=\bp M_1&0\\0&M_2\ep,
\qquad \Re M_1\ge \eta>0,
\quad 
\Re M_2\le -\eta<0.
$$
Making the change of variables $\cW=TZ$, we thus obtain
$$
\dot Z=(T^{-1}\cG T -T^{-1}\dot T)Z=
\Big(\bp M_1&0\\0&M_2\ep + o(1)\Big)Z
$$
for  $\tilde x\le -M$.
Applying the tracking lemma, Lemma \ref{reduction}, we find that
the manifold of solutions of \eqref{1linrNS}--\eqref{coeff}
decaying at $-\infty$
is within angle $o(1)$ of the unstable subspace of $\cG(\lambda, \tilde x)$
at each $\tilde x\le -M$, in particular for $\tilde x=-M$.

But, by the same reasoning, the manifold of decaying solutions of 
\eqref{sfirst} at $\tilde x=-M$ is also within angle $o(1)$ of
the unstable subspace of $\hat \cG$, which, by continuity in $\eps$/uniform
spectral gap is within angle $O(\eps)=o(1)=o(\tilde\lambda)$ 
of the unstable subspace of $\cG$.
From this, and \eqref{further+decayhatW},
it follows that, up to a normalizing factor $\hat \Psi(\eps, \tilde \lambda)$
that may be taken analytic in $\tilde \lambda$, 
\be\label{g1}
\frac{D_{rNS}(\lambda)}{ \hat \Psi(M,\eps, \lambda)}=
D_{NS}(\tilde \lambda ) +o(\tilde \lambda)
\ee
on $\Re \lambda >-\eta$, $ C^{-1}\le |\tilde \lambda|\le C$, 
uniformly as $M\to \infty$, and $\eps \to 0$, for $C>0$ arbitrary.

\subsubsection{Case b. $1/C\eps\ge |\lambda|\ge C>>1$}\label{IIb}
Finally, we treat the more delicate case 
$C\le |\lambda|\le 1/C\eps$, 
or $ C\eps\le |\tilde \lambda|\le C^{-1}$, with $C>>1$.
Proceeding as in Section \ref{slowzone}
by the series of coordinate transformations
\eqref{cZeq}--\eqref{Z1eq}, but taking account of the different
order of $\tilde \lambda$ in this case, in particular, noting
that $\eps m$ in \eqref{H} is now $O(\tilde \lambda)$ and 
not $O(\eps)$ as before, we obtain 
\be\label{cKII}
\cK^\eps = 
\begin{pmatrix} (\eps E-\tilde \lambda)\tilde A^{-1}+\tilde \lambda^2 \beta &0\\ 
0&N\\ \end{pmatrix}
+
\bp  O(\eps|\tilde \lambda| + |\tilde \lambda||\dot{\bar W}^\eps|)& 
O(\eps|\tilde \lambda| + |\tilde \lambda||\dot{\bar W}^\eps|)\\
\\ O(\eps + |\dot{\bar W}^\eps|) & O(\eps + |\dot{\bar W}^\eps|)\ep
\ee
in place of \eqref{cK}, 
where $\beta$ is the ``frozen-coefficient corrector'' obtained by dropping
terms involving $\tilde x$-derivatives of the transformations involved,
\be\label{SinvII}
S^{-1}=\bp I & -\Phi_1\\ -\Phi_2 & I\ep=
\bp I & |\tilde \lambda|O(\eps + |\dot{\bar W}^\eps|)\\ 
O(\eps + |\dot{\bar W}^\eps|) & I\ep
=
\bp I& o(|\tilde \lambda|)& o(|\tilde \lambda|)\\
 o(1)& I& o(1)\\
 o(1)&  o(1)& I\ep
\ee
for $\tilde x\ge -M$
in place of \eqref{Sinv}, and 
\be\label{fastexpII}
\cW_{n+1}^-= c(- M)
\bp o(1)|\tilde \lambda| v \\ (I+o(1))v 
\ep
\ee
in place of \eqref{fastexp}, where
$v$ lies in the unique stable eigendirection of $N(-M)$, with
$$
N:=\tilde b^{-1}(\tilde A_{22}- \tilde A_{21}\tilde A_{11}^{-1}\tilde A_{12}).
$$

By the same reasoning, applied to the viscous shock equations
\eqref{sfirst}, we have 
$$
\hat \cW_{n+1}^-=  c(- M) \bp o(1)|\tilde \lambda| \hat v \\ (I+o(1))\hat v \ep,
$$
where $\hat v$ lies in the same unique eigendirection.
and $\hat \cW_{n+1}$ is the single fast-decaying mode 
of \eqref{sfirst} at $-\infty$. Comparing, we thus have, for some
normalizing factor $\psi(M,\eps, \lambda)$,
analytic in $\lambda$,
\be\label{fastcor}
\cW_{n+1}^-(-M)=
\big(I+o(|\tilde \lambda|)\big)
\psi(M,\eps, \tilde \lambda)
\hat \cW_{n+1}^-( -M).
\ee

We now turn to the description of the remaining, slow-decaying, modes
$ \cW^-_{1},\dots \cW^-_{n}$.
Focusing on the region $\tilde x\ge -C_2|\log |\tilde \lambda||$,
$C_2>>1$, we find that the slow, first, equation of \eqref{decoup} becomes
\ba\label{Z1eqII}
\dot \cZ_1&=\Big((\eps E-\tilde \lambda) \tilde A^{-1} 
+\beta \tilde \lambda^2\Big) \cZ_1 + 
O(\eps +|\dot{\bar W}^\eps(x/\eps)|)|\tilde \lambda|\cZ_1 \\
&=
\Big((\eps E-\tilde \lambda) \tilde A^{-1} +\beta \tilde \lambda^2\Big) \cZ_1 
+ o(1)|\tilde \lambda|^2 \cZ_1, 
\ea
where
$\big((\eps E-\tilde \lambda) \tilde A^{-1} +\beta \tilde \lambda^2\big)$
agrees to second order in $\tilde \lambda <<1$ with the restriction of
$\cG$ to its slow subspace, which, recall, is the direct sum of 
$n$ slow-decaying modes as $\tilde x\to -\infty$ and a single slow-growing mode
as $\tilde x\to -\infty$.
Here, $o(1)\to 0$ as $C\to \infty$ and $\eps\to 0$.

Computing explicitly, and noting that $\bar z\sim e^{-\theta \eps |\tilde x|}$,
we have
$$
\eps E=
\eps \begin{pmatrix} qkd\phi(\bar u^\eps)\bar z  & qk\phi(\bar u) \\ 
-kd\phi(\bar u^\eps)\bar z & -k\phi(\bar u)\end{pmatrix}
=
\eps \begin{pmatrix} 0  & qk\phi(\bar u) \\ 0 & -k\phi(\bar u)\end{pmatrix}
+\alpha(x,\tilde \lambda),
$$
where
$\alpha=O(\eps e^{-\theta \eps |\tilde x|})$ is both $o(\tilde \lambda)$
and uniformly bounded in $L^1(\tilde x)$.
Thus, up to $O(\alpha)+o(|\tilde \lambda|^2)$ the eigenvalues of
$\big((\eps E-\tilde \lambda) \tilde A^{-1} +\beta \tilde \lambda^2\big)$
agree with the growth rates $\mu_j$ of slow modes 
$e^{\mu_j \tilde x}W_j$, $W_j=\const$,
of the nonreacting frozen-coefficient operator
\be\label{ccsys}
\tilde \lambda I+
 \begin{pmatrix}df(u)-sI & 0 \\ 0 & -s\end{pmatrix}\partial_{\tilde x}
+\begin{pmatrix}B(u) & 0 \\ 0 & C(u,z)\end{pmatrix}\partial_{\tilde x}^2
+
\begin{pmatrix} 0  & qk\phi(u) \\ 0 & -k\phi(u)\end{pmatrix}
\ee
obtained by neglecting $O( \partial_{\tilde x} \bar W^\eps)$
derivative terms and setting $\bar z$ to zero.

A standard low-frequency matrix perturbation computation
\cite{Z1,LyZ1,LyZ2,TZ4,LRTZ} shows that the modes $W_j$ 
of \eqref{ccsys}
 are analytic in
$\tilde \lambda$ and $\eps$, lying approximately in the eigendirections
of $\tilde A$, with the associated growing (i.e., negative real part) 
eigenvalue $\mu_{n+1}$ separated in modulus by order $|\tilde \lambda|$
from the decaying ones;\footnote{
This does not require strict hyperbolicity of
$df$, but only $\det(df-s)\ne 0$; see Remark \ref{semirmk}.}
moreover, by \eqref{H4}, 
the growing eigenvalue is separated in
 real part by order $\sim |\tilde \lambda|^2$
from associated decaying (positive real part) eigenvalues.

From this, and the fact that the entire coefficient matrix is order $\tilde \lambda$, it follows that there exists 
a smooth matrix-valued function  $Q(\bar W^\eps)$ with
$$
Q^{-1} \big((\eps E-\tilde \lambda) \tilde A^{-1} +\beta \tilde \lambda^2\big) Q=
\bp m_1+\alpha_1& 0\\0&m_2+\alpha_2\ep,
$$
\be\label{sep}
\Re m_1\ge \check \eta |\tilde \lambda|^2>0,
\quad 
\Re m_2\le -\check \eta |\tilde \lambda|^2<0,
\quad
|\alpha_j|_{L^1(\tilde x)}\le C,
\qquad \check \eta=\const.
\ee
Making the change of coordinates $\cZ_1= Qz$, and
noting that $Q'=O(\dot{\bar W}^\eps)$, we thus obtain
$$
\dot z=
\Big( Q^{-1} 
\big((\eps E-\tilde \lambda) \tilde A^{-1} +\beta \tilde \lambda^2\big) Q
-Q^{-1}\dot Q\Big)z
=
\bp m_1& 0\\0&m_2\ep z + \big( o(|\tilde \lambda|^2) 
+ O(\dot{\bar W}^\eps) \big)z,
$$
with gap condition \eqref{sep}.

Since $|\dot{\bar W}^\eps|= O(\eps)=o(\tilde \lambda)$, 
and $m_j$ are spectrally separated by modulus $\sim|\tilde \lambda|$,
there is a further smooth coordinate change $z=\cR y$ with
$\cR=I+o(1)$ converting the equations to
$$
\dot y=
\bp m_1+\check \alpha_1& 0\\0&m_2+\check \alpha_2\ep z + 
\big( o(|\tilde \lambda|^2) + 
\beta \big)z,
$$
$\check \alpha_j=\alpha_j + O(\dot{\bar W}^\eps)$,
$\beta= o(\dot{\bar W}^\eps) \big)$, 
where $|\check \alpha_j|_{L^1}\le C$, 
$|\beta|_{L^1(\tilde x)}|= o(1)$, and $m_j$ satisfy \eqref{sep}.
That is, we have an equation of form \eqref{blockdiag2} with
$\delta=o(|\tilde \lambda|^2)$ and $\eta\sim |\tilde \lambda|^2$,
so that $\delta/\eta=o(1)$.

Applying the tracking lemma, Lemma \ref{reduction},
as generalized in \eqref{Phibd2}, Remark \ref{L1}, 
at $\tilde x=-C|\log |\tilde \lambda||$, and untangling coordinate
changes, we thus find that the slow-decaying modes $\cW_{j}^-$,
$j=1,\dots, n$ lie within angle $o(1)$ of the stable subspace of
$\cG(-C|\log |\tilde \lambda||)$, which in turn lies within angle $o(1)$
of the stable subspace of $\hat G(\tilde \lambda, -\infty)$ and
(by a repetition of the same argument), analytic multiples of 
the slow-decaying modes $\hat \cW_j^-$, $j=1,\dots, n$, at $\tilde x=-M$.

Finally, going back to the original equation \eqref{Z1eqII}, and noting
that $\dot{\cZ}_1=O(|\tilde\lambda|)\cZ_1$, we find that the change
in $\cZ_1$ in evolving from $\tilde x=-C|\log |\tilde \lambda||$ to
$\tilde x=-M$ is order 
$$
\big(e^{|\tilde \lambda||\log |\tilde \lambda||}- 1\big)
|\cZ_1(-C |\log|\tilde \lambda||)|
\sim  |\tilde \lambda||\log |\tilde \lambda||
|\cZ_1(-C |\log|\tilde \lambda||)|
=o(1) |\cZ_1(-C |\log|\tilde \lambda||)|,
$$
hence at $\tilde X=-M$ also the slow-decaying modes $\cW_{j}^-$,
$j=1,\dots, n$ lie within angle $o(1)$ of analytic multiples of 
the slow-decaying modes $\hat \cW_j^-$ at $\tilde x=-M$.
Collecting facts, we have
\be\label{slowII}
\cW_{j}^-(-M)=(I+o(1))\psi_j(M,C_2,\eps, \tilde \lambda)\hat \cW_j^-(-M),
\qquad j=1, \dots, n,
\ee
where $\psi_j$ are nonvanishing and analytic in $\tilde \lambda$.
(Here, we are using also the fact that, by \eqref{cKII}--\eqref{SinvII},
the manifold of all slow modes, both growing and decaying,
stays angle $o(1)$ close to the 
slow subspace of $\cG(\tilde \lambda, \tilde x)$ for all $\tilde x\ge -M$.)

Finally, collecting estimates \eqref{further+decayhatW},
\eqref{fastcor},
and \eqref{slowII},
we see that fast modes $\cW_j^\pm$, up to nonvanishing analytic
factor $\psi_j$, are given by $(I+o(\tilde \lambda))$ times
the correspong fast modes $\hat W_j^\pm$ of the associated
shock stability problem, while slow modes $\cW_j^-$ are
are given by $(I+o(1))$ times
the correspong slow modes $\hat W_j^-$.
Recalling, similarly as in estimate \eqref{basicdet}, that
at $\tilde \lambda=0$, the fast mode $\hat \cW_{n+1}^-$
is a linear combination of the fast modes $\hat \cW_{j}^+$,
so that at $\tilde \lambda$ it remains within angle $O(\tilde \lambda)$
of their combination, we find, applying a column operation
cancelling $\hat \cW_{n+1}^-$ to order $\tilde \lambda$ and factoring
out $\tilde \lambda$ from that column, that we reduce the (rNS)
determinant \eqref{evanseqrNS} to
\be\label{estagain}
\tilde \lambda (1+o(1))(\Pi_j \hat \psi_j)D_{NS}(\tilde \lambda).
\ee
The key observation here is that 
because only fast modes are involved in the vanishing
of $D_{NS}$ at $\tilde \lambda=0$, {\it only fast modes must
be estimated to the sharper relative error $o(\tilde \lambda)$}
in order to obtain the result \eqref{estagain}, with $o(1)$
tolerance sufficing for slow modes.

Thus, we obtain again an estimate
$\frac{D_{rNS}(\lambda)}{ \hat \Psi(M,C,\eps, \lambda)}=
D_{NS}(\tilde \lambda ) +o(\tilde \lambda)$
on $\Re \lambda >-\eta$, $ C\eps \le |\tilde \lambda|\le C^{-1}$, 
where $o(1)\to 0$ uniformly as $M\to \infty$, $C\to \infty$, 
and $\eps \to 0$, with 
$\hat \Psi:=\Pi_j \hat \psi_j$ nonvanishing and analytic in $\tilde \lambda$.
Note that we are using here the assumption that $C>>1$, which
was not needed in case a.

\subsection{Convergence to $D_{NS}$}\label{Nconv}
\bpr\label{pII}
Assuming (H0)--(H3), (P1)--(P3),
For $C/\eps\ge |\lambda|\ge C>>1$, $\Re \lambda \ge 0$,
$\frac{D^\eps_{rNS}(\lambda)}{\eps \lambda \check \Psi(\lambda, \eps)}$
converges uniformly as $\eps \to 0$ to 
$\frac{D_{NS}(\eps\lambda)}{\eps\lambda}$,
where $\check \Psi(\cdot,\cdot)$ is 
a nonvanishing factor that is analytic in $\lambda$.
\epr

\begin{proof}
From \eqref{g1}, \eqref{estagain}, we have in either case a or b that
\be\label{last}
\frac{D_{rNS}(\lambda)}{ \hat \Psi(M,C,\eps, \lambda)}=
D_{NS}(\eps \lambda ) +o(1)\eps \lambda
\ee
on $\Re \lambda >-\eta$, uniformly as $\eps \to 0$.
where $o(1)\to 0$ uniformly as $M\to \infty$, $C\to \infty$, 
and $\eps \to 0$.
Choosing monotone increasing sequences $C_j$, $M_j$ and a monotone
decreasing sequence $\eps_j$ such that $o(1)\le 1/j$ in \eqref{last},
define $\check \Psi(\lambda,\eps)$ to be equal to the function
 $\hat\Psi(\lambda,\eps)$ in \eqref{last} that is associated
with $C_j$, $M_j$, where $j$ is the maximum integer such that
$\eps \le \eps_j$.
Then, $\hat \Psi$ is analytic in $\lambda$ by construction, and,
combining \eqref{last} with the definition of
$\eps_j$, $M_j$, $C_j$, we have
$
\Big|
\frac{D^\eps_{rNS}(\lambda)}{\eps \lambda\hat \Psi(\lambda, \eps)}-
\frac{D_{NS}(\eps \lambda)}{\eps \lambda}
\Big| \le C/j \to 0$
as $\eps \to 0$, giving the result.
\end{proof}

\bc\label{sII}
Assuming (H0)--(H3), (P1)--(P3),
for $C\le |\lambda| \le C/\eps$, $C$ sufficiently large,
on $\Re \lambda>-\eta$ for $\eta$, $\eps>0$ sufficiently small, 
$\eps$ times each zero of $D^\eps_{rNS}$ converges
to a zero of $\frac{D_{NS}(\tilde \lambda)}{\tilde \lambda}$ 
on $\Re \tilde \lambda \ge 0$;
moreover, each zero of $D_{NS}$ on $\Re \tilde \lambda>0$ 
is the limit of $\eps$ times a zero of $D_{rNS}$  
on $\Re \lambda>0$, for $C\le |\lambda| \le C/\eps$.
\ec

\begin{proof}
Noting that zeros of $D^\eps_{rNS}$ agree with zeros of
$\frac{D^\eps_{rNS}(\lambda)}{\eps\lambda \hat \Psi(\lambda, \eps)}$,
we obtain the result by Proposition \ref{pII}
and properties of uniform limits of analytic functions.
\end{proof}

\section{Region III: $|\lambda|\ge C/\eps$, $C>>1$}\label{III}
Finally, we consider the straightforward ``hyperbolic--parabolic'' region
$|\lambda|\ge C/\eps$, $C>>1$, or $|\tilde \lambda |>>1$,
on which zeros of $D_{rNS}$ are prohibited stable by 
basic hyperbolic--parabolic structure/well-posedness of
the underlying problem \eqref{rNS}.

\bpr\label{pIII}
$D_{rNS}$ does not vanish for
$|\lambda|\ge C/\eps$, $C>>1$, $\Re \lambda \ge 0$.
\epr

\begin{proof}
In fast coordinates $\tilde x$,  $|\tilde \lambda|>>1$,
this follows by the same high-frequency analysis used in \cite{MaZ3}
to treat the viscous shock case, based on
the tracking/reduction lemma, Lemma \ref{reduction}.
See Propositon 5.2, \cite{MaZ3}, or Proposition 4.33, \cite{Z2}.
\end{proof}





\appendix
\section{Asymptotic ODE theory}\label{asymptotic}

\subsection{The conjugation lemma}\label{s:conj}

Consider a general first-order system
\be \label{gen1}
W'=A^p(x,\lambda)W
\ee
with asymptotic limits $A^p_\pm$ as $x\to \pm \infty$,
where $p\in \RR^m$ denote model parameters.
\bl [\cite{MeZ1,PZ}]\label{conjlem}
Suppose for fixed $\theta>0$ and $C>0$ that 
\be\label{udecay}
|A^p-A^p_\pm|(x,\lambda)\le Ce^{-\theta |x|}
\ee
for $x\gtrless 0$ uniformly for $(\lambda,p)$ in a neighborhood of 
$(\lambda_0)$, $p_0$ and that $A$ varies analytically in $\lambda$ 
and smoothly (resp. continuously) in $p$ 
as a function into $L^\infty(x)$.
Then, there exist in a neighborhood of $(\lambda_0,p_0)$
invertible linear transformations $P^p_+(x,\lambda)=I+\Theta_+^p(x,\lambda)$ 
and $P_-^p(x,\lambda) =I+\Theta_-^p(x,\lambda)$ defined
on $x\ge 0$ and $x\le 0$, respectively,
analytic in $\lambda$ and smooth (resp. continuous) in $p$ 
as functions into $L^\infty [0,\pm\infty)$, such that
\begin{equation}
\label{Pdecay} 
| \Theta_\pm^p |\le C_1 e^{-\bar \theta |x|}
\quad
\text{\rm for } x\gtrless 0,
\end{equation}
for any $0<\btheta<\theta$, some $C_1=C_1(\bar \theta, \theta)>0$,
and the change of coordinates $W=:P^p_\pm Z$ reduces \eqref{gen1} to 
the constant-coefficient limiting systems
\begin{equation}
\label{glimit}
Z'=A^p_\pm Z 
\quad
\text{\rm for } x\gtrless 0.
\end{equation}
\el

\begin{proof}
The conjugators $P^p_\pm$ are constructed by a
fixed point argument \cite{MeZ1}
as the solution of an integral equation
corresponding to the homological equation
\be\label{homolog}
P'=A^pP-A^p_\pm P.
\ee
The exponential decay \eqref{udecay} is needed to 
make the integral equation contractive with respect to $L^\infty[M,+\infty)$ 
for $M$  sufficiently large.  Continuity of $P_\pm$ with respect 
to $p$ (resp. analyticity with respect to $\lambda$) then follow
by continuous (resp. analytic) dependence on parameters of 
fixed point solutions.
Here, we are using also the fact that \eqref{udecay} plus continuity
of $A^p$ from $p\to L^\infty$ together imply continuity of
$e^{\tilde \theta |x|}(A^p-A^p_\pm)$ from $p$ into $L^\infty[0,\pm\infty)$
for any $0<\tilde \theta < \theta$, in order to obtain the needed
continuity from $p\to L^\infty$ of the fixed point mapping.
See also \cite{PZ,GMWZ5}.
\end{proof}


\br\label{specialform}
In the case that $A$ is block diagonal or
triangular, the conjugators $P_\pm$ may evidently be taken block diagonal
or triangular as well, by carrying out the same fixed-point argument
on the invariant subspace of \eqref{homolog} consisting
of matrices with this special form.
This can be of use in problems with multiple scales; 
see, for example, Thm 1.16, \cite{BHZ}.
\er

\begin{definition}[Abstract Evans function]\label{abevans}
Suppose that on the interior of a set $\Omega$ in $\lambda$, $p$, the 
dimensions of the stable and unstable subspaces of $A_\pm^p(\lambda)$
remain constant, and agree at $\pm \infty$ 
(``consistent splitting'' \cite{AGJ}), and that these subspaces
have analytic bases $R_j^\pm$ extending
continuously to boundary points of $\Omega$.  Then, 
the Evans function is defined on $\Omega$ as
\be\label{eq:evans}
\begin{aligned}
D^p(\lambda)&:=
\det( P^+R_1^+,P^+R_2^+, P^-R_1^-, P^-R_2^-)|_{x=0},\\
\end{aligned}
\ee
where $P_\pm^p$ are as in Lemma \ref{conjlem}.
\end{definition}

\subsection{The convergence lemma}\label{s:conv}
Consider a family of first-order equations 
\be \label{gen2}
W'=A^p(x,\lambda)W
\ee
indexed by a parameter $p$, and satisfying exponential
convergence condition \eqref{udecay} uniformly in $p$.
Suppose further that
\begin{equation} \label{residualest}
|(A^p- A^p_\pm)-
(A^0- A^0_\pm)|
\le C|p|e^{-\theta |x|}, \qquad \theta>0
\end{equation}
and
\begin{equation} \label{newest}
|( A^p- A^0)_\pm)| \le C|p|.
\end{equation}

\begin{lemma}[\cite{PZ,BHZ}] \label{evanslimit}
Assuming \eqref{udecay} and \eqref{residualest}--\eqref{newest},
for $|p|$ sufficiently small,
there exist invertible linear transformations 
$P_+^p(x,\lambda)=I+\Theta_+^p(x,\lambda)$ 
and $P_-^0(x,\lambda) =I+\Theta_-^p(x,\lambda)$ defined
on $x\ge 0$ and $x\le 0$, respectively,
analytic in $\lambda$ as functions into $L^\infty [0,\pm\infty)$, such that
\begin{equation}
\label{cPdecay} 
| (P^p-P^0)_\pm |\le C_1 |p| e^{-\bar \theta |x|}
\quad
\text{\rm for } x\gtrless 0,
\end{equation}
for any $0<\btheta<\theta$, some $C_1=C_1(\bar \theta, \theta)>0$,
and the change of coordinates $W=:P_\pm^p Z$ reduces \eqref{gen2} to 
the constant-coefficient limiting systems
\begin{equation}
\label{cglimit}
Z'=A^p_\pm(\lambda) Z 
\quad
\text{\rm for } x\gtrless 0.
\end{equation}
\end{lemma}

\begin{proof}
Applying the conjugating transformation $W\to
(P^0_+)^{-1}W$ for the $p=0$ equations, we may reduce to the
case that $A^0$ is constant, and $P^0_+\equiv I$, noting that
the estimate \eqref{residualest} persists under well-conditioned
coordinate changes $W=QZ$, $Q(\pm \infty)=I$,
transforming to 
\be\label{calc1}
\begin{aligned}
&|\big(Q^{-1}A^pQ-Q^{-1}Q'-A^p_\pm\big) -
\big(Q^{-1}A^0Q-Q^{-1}Q'-A^0_\pm\big)|\\
&\qquad \le |Q\big((A^p-A^p_\pm)-(A^0-A^0_\pm)\big)Q^{-1}|
+ |Q^{-1}(A^p-A^0)_\pm Q-(A^p-A^0)_\pm |,
\\
\end{aligned}
\ee
where
\be\label{calc2}
|Q^{-1}(A^p-A^0)_\pm Q-(A^p-A^0)_\pm |= O(|Q-I|)|(A^p-A^0)_\pm|
= O(e^{-\theta|x|})|p|.
\ee
In this case,
\eqref{residualest} becomes just
$$ |A^p- A^p_\pm| \le
C_1|p|e^{-\theta |x},
$$
and we obtain directly from the conjugation lemma, Lemma
\ref{conjlem}, the estimate
$$
|P^p_+ - P^0_+|=
|P^p_+ - I|\le CC_1|p| e^{-\bar \theta|x|}
$$
for $x>0$, and similarly for $x<0$, verifying the result.
\end{proof}

\br\label{varconj}
In the case $A^p_\pm \equiv \const$, or, equivalently, for which
\eqref{residualest} is replaced by $ |A^p-A^0| \le C_1|p|e^{-\theta |x}, $
we find that the change of coordinates
$W=\tilde P^p_\pm Z$, $\tilde P^p_\pm :=(P^0)_\pm ^{-1}P_\pm ^p$, 
converts \eqref{gen2}
to $Z'=A^0Z$, where $\tilde P^p_\pm =I+\tilde \Theta^p_\pm$ with
\be\label{altconv}
|\tilde \Theta^p_\pm|\le CC_1|p| e^{-\bar \theta|x|}
\quad \hbox{\rm for $x\gtrless 0$}.
\ee
That is, we may conjugate not only to constant-coefficient equations,
but also to exponentially convergent variable-coefficient equations,
with sharp rate \eqref{altconv}.

In the general case $A^p\pm\ne A^0\pm$, we may 
still conjugate \eqref{gen2} to $Z'=A^0Z$ by the 
change of coordinates
$W=\hat P^p_\pm Z$, $\hat P^p_\pm :=(P^0)_\pm ^{-1}Q_\pm P_\pm ^p$, 
where $Q_\pm$ defined by
\be\label{Qdef}
Q'=A^p_\pm Q- QA^0_\pm
\ee
conjugates the constant-coefficient equation $Y'=A^p\pm Y$ to 
$X'=A^0_\pm X$, obtaining bounds
\be\label{altconv2}
\hat P^p_\pm= I+\tilde \Theta^p_\pm, 
\qquad
|\tilde \Theta^p_\pm|\le CC_1|p|
\quad \hbox{\rm for }\; |x|\le C
\ee
valid for {\rm finite values of $x$}.
However, in general, $Q_\pm$ grow without bound as $x\to \pm \infty$.
\er

\br\label{convrmk}
As observed in \cite{PZ},
provided that the stable/unstable subspaces of $A^p_+$/$A^p_-$
converge to those of $A^0_+$/$A^0_-$, as typically holds given
\eqref{newest}-- in particular, this holds by standard matrix
perturbation theory \cite{K} if the stable and unstable eigenvalues
of $A^0_\pm$ are spectrally separated-- 
\eqref{cPdecay} gives immediately uniform convergence
of the Evans functions $D^p$ to $D^0$ on compact sets of $\Omega$,
by definition \eqref{abevans}.
\er

\subsection{The tracking lemma}\label{s:track}
Consider an approximately block-diagonal system
\begin{equation}
W'= \bp M_1 & 0 \\ 0 & M_2 \ep(x,p) + \delta(x,p) \Theta(x,p) W,
\label{blockdiag}
\end{equation}
where $\Theta$ is a uniformly bounded matrix, $\delta(x)$ scalar, 
and $p$ a vector of parameters,
satisfying a pointwise spectral gap condition
\begin{equation}
\min \sigma(\Re M_1^\varepsilon)- \max \sigma(\Re M_2^\varepsilon)
\ge \eta(x)
\, \text{\rm for all } x.
\label{gap}
\end{equation}
(Here as usual $\Re N:= (1/2)(N+N^*)$ denotes the
``real'', or symmetric part of $N$.)
Then, we have the following
{\it tracking/reduction lemma} of \cite{MaZ3,PZ}.

\begin{lemma}[\cite{MaZ3,PZ}] \label{reduction}
Consider a system \eqref{blockdiag} under the gap assumption
\eqref{gap}, with $\Theta^\varepsilon$ uniformly bounded and
$\eta\in L^1_{\rm loc}$.
If $\sup (\delta/\eta)(x)$ is sufficiently small,
then there exist (unique) linear
transformations $\Phi_1(x,p)$ and
$\Phi_2(x,p)$, possessing the same
regularity with respect to $p$
as do coefficients $M_j$ and
$\delta\Theta$, for which the graphs $\{(Z_1,
\Phi_2 Z_1)\}$ and $\{(\Phi_1 Z_2,Z_2)\}$ are
invariant under the flow of \eqref{blockdiag}, and satisfy
\be\label{Phibd}
\sup|\Phi_1|, \, \sup|\Phi_2| \le C \sup(\delta/\eta) 
\ee
and
\ba\label{ptwise}
|\Phi^\varepsilon_1(x)|&\le
C\int_x^{+\infty} e^{\int_y^x \eta(z)dz} \delta(y)dy,
\qquad
|\Phi^\varepsilon_1(x)| \le C
\int_{-\infty}^{x} e^{\int_y^x -\eta(z)dz} \delta(y)dy.
\ea
\end{lemma}

\begin{proof}
By the change of coordinates $x\to \tilde x$, $\delta \to \tilde \delta:=
\delta/\eta$ with
$d\tilde x/dx=\eta(x)$,
we may reduce to the case $\eta\equiv {\rm constant}= 1$ treated in \cite{MaZ3}.
Dropping tildes and setting $\Phi_2:= \psi_2\psi_1^{-1}$, where 
$(\psi_1^t,\psi_2^t)^t$ satisfies \eqref{blockdiag}, 
we find after a brief calculation that $\Phi_2$ satisfies
\be
\Phi_2'= 
(M_2 \Phi_2 - \Phi_2 M_1) + \delta Q(\Phi_2),
\label{Phieqn}
\ee
where $Q$ is the quadratic matrix polynomial
$
Q(\Phi):=
\Theta_{21} + \Theta_{22}\Phi - \Phi \Theta_{11} + \Phi \Theta_{12} \Phi.
$
Viewed as a vector equation, this has the form
\be
 \Phi_2'= \cM \Phi_2 + \delta Q(\Phi_2),
\label{vectoreqn}
\ee
with linear operator
$\cM \Phi:= M_2 \Phi - \Phi M_1$.
Note that a basis of solutions of the decoupled equation
$ \Phi'= \cM \Phi$
may be obtained as the tensor product $\Phi=\phi \tilde \phi^*$
of bases of solutions of $\phi'=M_2 \phi$ and
$\tilde \phi'= -M_1^* \tilde \phi$, whence we obtain from
\eqref{gap} 
\be
e^{\cM z}\le Ce^{-\eta z}, \quad \hbox{\rm for }\; z>0,
\label{expbd}
\ee
or uniform exponentially decay in the forward direction.

Thus, assuming only that $\Phi_2$ is bounded at $-\infty$, we obtain
by Duhamel's principle the integral fixed-point equation
\be
\Phi_2(x)= \CalT \Phi_2(x):=
\int_{-\infty}^x e^{\cM (x-y)} \delta(y)Q(\Phi_2)(y)
\,dy.
\label{inteqn}
\ee
Using \eqref{expbd}, we find that $\CalT$ is a contraction
of order $O(\delta/\eta)$, hence \eqref{inteqn} determines
a unique solution for $\delta/\eta$ sufficiently small, which,
moreover, is order $\delta/\eta$ as claimed.
Finally, substituting 
$Q(\Phi)=O(1+|\Phi|^2)=O(1)$ in 
\eqref{inteqn}, we obtain
$$
|\Phi_2(x)|\le 
C\int_{-\infty}^x e^{\eta (x-y)} \delta(y)
\,dy
$$
in $\tilde x$ coordinates, or, in the original $x$-coordinates,
\eqref{ptwise}.
A symmetric argument establishes existence of $\Phi_1$ with 
the asserted bounds.
Regularity with respect to parameters is inherited as usual
through the fixed-point construction via the Implicit Function Theorem.
\end{proof}

\br
For $\eta$ constant and $\delta$ decaying at exponential rate
strictly slower that $e^{-\eta x}$ as $x\to +\infty$, 
we find from \eqref{ptwise} that $\Phi_2(x)$ decays like $\delta/\eta$
as $x\to +\infty$,
while if $\delta(x)$ merely decays monotonically as $x\to -\infty$, we
find that $\Phi_2(x)$ decays like $(\delta/\eta)$ as $x\to -\infty$,
and symmetrically for $\Phi_1$.
\er

\br\label{tritrack}
1. A closer look at the proof of Lemma \ref{reduction} shows
that, in the approximately block lower-triangular case,
$\delta \Theta_{21}$ not necessarily small,
there exists a {\rm block-triangularizing} transformation
$\Phi_1=O(\sup|\delta/\eta|)<<1$, under the much less restrictive conditions
$$
\sup \Big(|\delta/\eta|(| \Theta_{11}| +| \Theta_{22}|) \Big)
<1
\; \hbox{\rm and } \;
\sup(|\delta/\eta| |\Theta_{21}|) << \frac{1}{\sup|\delta/\eta|}.
$$

2. Similarly, in the standard, approximately block-diagonal case,
an examination of the proof shows that bounds \eqref{Phibd} 
may be sharpened to
\ba\label{twoPhibds}
\sup|\Phi_1|&\le C \sup(\delta/\eta) 
\Big( \sup|\Theta_{12}|+
\sup(\delta/\eta) \sup(|\Theta_{11}|+ |\Theta_{22}|)
+ \sup(\delta/\eta)^2 \sup|\Theta_{21}| \Big)
,\\
\sup|\Phi_2| &\le C \sup(\delta/\eta) 
\Big( \sup|\Theta_{21}|+
\sup(\delta/\eta) \sup(|\Theta_{11}|+ |\Theta_{22}|)
+ \sup(\delta/\eta)^2 \sup|\Theta_{12}| \Big).
\ea
\er

\br\label{L1}
An important observation of \cite{MaZ3,PZ} is that
hypothesis \eqref{gap} of Lemma \ref{reduction} may be weakened to
\begin{equation}
\min \sigma(\Re M_1^\varepsilon)- \max \sigma(\Re M_2^\varepsilon)
\ge \eta(x) +\alpha(x,p)
\label{gapL1}
\end{equation}
with no change in the conclusions,
for any $\alpha$ satisfying a uniform $L^1$ bound 
$|\alpha(\cdot,p)|_{L^1}\le C_1$.
(Substitute
$e^{\cM x}\le Ce^{C_1}e^{-\eta z}$
for \eqref{expbd}, with no other change in the proof.)
More generally,
\eqref{blockdiag2} may be replaced in Lemma \ref{reduction} by
\begin{equation}
W'= \bp M_1+\alpha_1 & 0 \\ 0 & M_2+\alpha_2 \ep(x,p) + 
(\delta \Theta +\beta)(x,p)W,
\label{blockdiag2}
\end{equation}
where $|\alpha_j|_{L^1}$ are bounded and $|\beta|_{L^1}$ 
is sufficiently small, with the conclusion that
\be\label{Phibd2}
\sup|\Phi_1|, \, \sup|\Phi_2| \le C( \sup(\delta/\eta) +|\beta|_{L^1}).
\ee
(The additional term $\int_{-\infty}^x e^{\cM (x-y)}Q_\beta(\Phi_2)(y) \,dy$,
$Q_\beta(\Phi_2):=
\beta_{21} + \beta_{22}\Phi - \Phi \beta_{11} + \Phi \beta_{12} \Phi$,
now appearing in the righthand side of \eqref{inteqn} is
contractive for $|\beta|_{L^1}$ small.)
These allow us to neglect commutator terms in some of the
more delicate applications of tracking: for example, the
high-frequency analysis of \cite{MaZ3}, or the analysis of
case IIb in Section \ref{IIb}.
\er


\end{document}